\documentclass[a4paper,10pt]{article}

\usepackage{amsmath, amsthm}
\usepackage[bookmarks=true]{hyperref}
\usepackage{bookmark}
\usepackage{mathrsfs}
\usepackage{amsfonts,amssymb,amscd}
\usepackage{indentfirst,graphics,epsfig,psfrag}
\usepackage{epic,bez123}

\newtheorem{theorem}{Theorem}[section]
\newtheorem{lemma}[theorem]{Lemma}
\newtheorem{cor}[theorem]{Corollary}
\newtheorem{prop}[theorem]{Proposition}

\numberwithin{equation}{section}

\newcommand{\N}{\hbox{\bf N}}

 \DeclareMathOperator{\ord}{ord}

 \DeclareMathOperator{\supp}{supp}

\renewcommand{\t}{\, | \,}

\begin{document}

\title{On the Erd{\H{o}}s-Ginzburg-Ziv constant of groups of the form $C_2^r\oplus C_n$}

\author{Yushuang Fan\\
\small Mathematical College \\[-0.8ex]
\small China University of Geosciences (Beijing)\\[-0.8ex]
\small Beijing, P. R. China\\
\small\tt fys@cugb.edu.cn\\
\and
\qquad
Qinghai Zhong\thanks{Corresponding author: Qinghai Zhong}\\
\small Institute for Mathematics and Scientific Computing\\[-0.8ex]
\small University of Graz\\[-0.8ex]
\small Heinrichstra{\ss}e 36, 8010 Graz, Austria\\
\small\tt qinghai.zhong@uni-graz.at
}
\date{}
\pdfbookmark{\contentsname}{Contents}
\maketitle

\begin{abstract}  
Let $G$ be a finite abelian group. The Erd{\H{o}}s-Ginzburg-Ziv constant $\mathsf s(G)$ of $G$ is defined as the smallest integer $l\in \mathbb{N}$ such that every sequence $S$ over $G$ of length $|S|\geq l$ has a  zero-sum subsequence $T$ of length $|T|= {\exp}(G)$. The value of this classical invariant for groups with rank at most two is known. But the precise value of $\mathsf s(G)$ for the groups of rank larger than two is difficult to determine. In this paper we pay our attentions to the groups of the form $C_2^{r-1}\oplus C_{2n}$, where  $r\geq 3$ and $n\ge 2$. We give a new upper bound of $\mathsf s(C_2^{r-1}\oplus C_{2n})$ for odd integer $n$. For $r\in [3,4]$, we  obtain that $\mathsf s(C_2^2\oplus C_{2n})=4n+3$ for $n\ge 2$ and  $\mathsf s(C_2^{3}\oplus C_{2n})=4n+5$ for $n\geq 36$.
\end{abstract}
\medskip

\smallskip


\section{ Introduction }

Let $G$ be a finite abelian group. We define some central invariants in zero-sum theory which have been studied since the 1960s. Let
\begin{itemize}
\item[$\bullet$] $\mathsf D(G)$ denote the smallest integer $l\in \mathbb{N}$ such that every sequence $S$ over $G$ of length $|S|\geq l$ has a nonempty zero-sum subsequence.
    
\item[$\bullet$] $\mathsf s(G)$ denote the smallest integer $l\in \mathbb{N}$ such that every sequence $S$ over $G$ of length $|S|\geq l$ has a zero-sum subsequence $T$ of length $|T|=\exp (G)$.
    
\item[$\bullet$] $\eta(G)$ denote the smallest integer $l\in \mathbb{N}$ such that every sequence $S$ over $G$ of length $|S|\geq l$ has a nonempty zero-sum subsequence $T$ of length $|T| \in [1, \exp (G)]$.
\end{itemize}

$\mathsf D(G)$ is called the \emph{Davenport constant} of $G$ and $\mathsf s(G)$ the \emph{Erd{\H{o}}s-Ginzburg-Ziv} (EGZ) constant of $G$. For the historical development of the field and the contributions of many authors we refer to the surveys \cite{GG, GE}. Here we can only provide a brief summary. There is the following chain of inequalities (\cite[5.7.2 and 5.7.4]{Ge-HK06a})
\begin{equation} \label{LOWETA}
\mathsf D (G) \le \eta (G) \le \mathsf s (G) - \exp (G) + 1 \le |G| \,.
\end{equation}
Clearly,  equality holds throughout for cyclic groups which implies  the classical Theorem of Erd{\H{o}}s-Ginzburg-Ziv dating back to 1961 and stating that  $\mathsf s (G) = 2 |G|-1$ (\cite{EGZ}). Since about ten years  the precise value of all three invariants is known for groups having rank at most two. We have (\cite[Theorem 5.8.3]{Ge-HK06a})

\medskip
\textbf{Theorem A.} {\it Let \ $G = C_{n_1} \oplus C_{n_2}$ with $1 \le n_1 \t n_2$. Then
\[
\mathsf s (G) = 2n_1 + 2n_2 - 3\,, \ \ \eta (G) = 2 n_1 + n_2 - 2 \,, \quad \text{and} \quad \mathsf D(G) = n_1 + n_2 - 1   \,.
\]
} 

\medskip

In groups of higher rank precise values for any of the three invariants are known only in very special cases. We briefly sketch the state of knowledge with a focus on groups of the form $C_2^r \oplus C_n$, where $r, n \in \mathbb N$, which have found special attention in all these investigations.

To begin with the Davenport constant, suppose that $G \cong C_{n_1} \oplus \ldots \oplus C_{n_r}$, where $1 < n_1 \mid \ldots \ldots \mid n_r$, and set $\mathsf D^* (G) = 1 + \sum_{i=1}^r(n_i-1)$. An example shows that $\mathsf D^* (G) \le \mathsf D (G)$, and equality holds for $p$-groups. It is open whether or not equality holds for groups of rank three (for recent progress see \cite{Bh-SP07a}). Not even the special case where $G=C_n^3$ is known, but we know that $\mathsf D^* (C_2 \oplus C_{2n}^4) < \mathsf D (C_2 \oplus C_{2n}^4)$ \cite[Theorem 3.1]{Ge-Li-Ph12}. If $n \ge 3$ is odd and $r \in \mathbb N$, then $\mathsf D^* (C_2^r \oplus C_{2n})=\mathsf D (C_2^r \oplus C_{2n})$ if and only if $r \le 3$. If $r \le 2$, then  the structure of the minimal zero-sum sequences over $C_2^3 \oplus C_{2n}$ is well-known (\cite{Sc11b}).
The only groups $G$ with $\mathsf D^* (G) < \mathsf D (G)$, for which the precise value of $\mathsf D (G)$ is known are groups of the form $G=C_2^4 \oplus C_{2n}$ for odd $n \ge 70$ (\cite[Theorem 5.8]{Sa-Ch14a}).

\smallskip
A simple example shows that $\eta (G) \le \mathsf s (G) - \exp (G) + 1$, and the standing conjecture (due to Weidong Gao \cite{Ga03a}) states that equality holds for all groups $G$. This has been confirmed for a variety of groups (see \cite{FYS} without knowing the precise value of $\eta (G)$ or $\mathsf s (G)$). If $G$ is an elementary $2$-group, then it can be seen right from the definitions that $\eta (G)=|G|$ and that $\mathsf s (G)=|G|+1$. If $G=C_3^r$, then $(\mathsf s (G)-1)/2$ is the maximal size of a cap in the $r$-dimensional affine space over $\mathbb F_3$ (\cite[Lemma 5.2]{Edel2007}). This invariant has been studied in finite geometry since decades, but precise values are known only for $r \le 6$ (\cite{Po08a}). For arbitrary primes $p$, the precise values of $\eta (C_p^3)$ and $\mathsf s (C_p^3)$ are unknown. However, there are standing conjectures which have been verified in very special cases (see \cite{G-H-S-T07}), and also the structure of sequences of length $\mathsf D (G)-1$ (resp. $\eta (G)-1$ or $\mathsf s (G)-1)$ that do not have a zero-sum subsequence (of the required length) has been studied for groups of the form $C_n^r$ (\cite[Theorem 3.2]{Ga-Ge-Sc07a}). For recent precise results for $\eta (G)$ and $\mathsf s (G)$ we refer the reader to \cite{Sc-Zh10a} and to \cite{FGZ}.

In the present paper we focus  on the EGZ constant $\mathsf s (G)$ and on $\eta (G)$ for   groups of the form $C_2^{r-1}\oplus C_{2n}$, where $n\ge 2$ is an integer. Our first result provides the best upper bound on $\mathsf s(C_2^r\oplus C_n)$ for $n\ge 3$ odd, which is known so far.

\begin{theorem}\label{Th2} $\mathsf s(C_2^{r-1}\oplus C_{2n})\leq 4n+2^r-5$ where $r\geq 3$ is a positive integer and $n\geq 3$ is an odd integer.
\end{theorem}
\medskip

If $G = C_2^{r-1} \oplus C_{2n}$ and $r \in [3,4]$, then we can provide precise results.

\begin{theorem} \label{Th1}
Let $n \ge 2$. 
\begin{enumerate}
\item $\eta(C_2^2\oplus C_{2n})=2n+4$ \ and \ $\mathsf s(C_2^2\oplus C_{2n})=4n+3$.

\smallskip
\item $\eta(C_2^3\oplus C_{2n})=2n+6$, \ and if $n \ge 36$, then \ $\mathsf s(C_2^3\oplus C_{2n})=4n+5$.
\end{enumerate}
\end{theorem}

\section{ Notations and Terminology}
Our notations and terminology are consistent with \cite{GG} and \cite{GE}. Let $\mathbb{N}$ denote the set of positive integers, $\mathbb{P}\subseteq \mathbb{N}$ the set of prime numbers and $\mathbb{N}_0=\mathbb{N}\cup\{0\}$. For real numbers $a, b\in \mathbb{R}$, we set $[a, b]=\{x\in \mathbb{Z}\mid a\leq x\leq b\}$.  Throughout this paper, all abelian groups will be written additively, and for $n\in \mathbb{N}$, we denote by $C_n$ a cyclic group with $n$ elements.

Let $G$ be a finite abelian group. We know that $|G|=1$ or $G\cong C_{n_1}\oplus \cdots \oplus C_{n_r}$ with $1<n_1\mid \cdots \mid n_r\in \mathbb{N}$, where $r=\mathsf{r}(G)\in \mathbb{N}$ is the \emph{rank} of $G$ and $n_r={\exp}(G)$ is the \emph{exponent} of $G$. We denote $|G|$ the \emph{cardinality} of $G$, and $\mathsf{ord}(g)$ the \emph{order} of elements $g\in G$. For convenience, denote $C_n^r=C_{n_1}\oplus \cdots \oplus C_{n_r}$ if $n_1=\cdots =n_r=n\in \mathbb{N}$.

Let $\mathscr{F}(G)$ be the free abelian monoid, multiplicatively written, with basis $G$. The elements of $\mathscr{F}(G)$ are called \emph{sequences }over $G$. A  sequence $S\in \mathscr{F}(G)$ will be written in the form
$$S=g_1\cdot \ldots \cdot g_l=\prod_{g\in G}g^{\mathsf v_{g}(S)}$$
with $\mathsf v_{g}(S)\in \mathbb{N}_0$ for all $g\in G$. We call $\mathsf v_{g}(S)$ the multiplicity of $g$ in $S$, and if $\mathsf v_{g}(S)>0$ we say that $S$ contains $g$. If for all $g\in G$ we have $\mathsf v_{g}(S)=0$, then we call $S$ the \emph{empty sequence} and denote $S=1\in \mathscr{F}(G)$. A sequence $S$ is called squarefree if $\mathsf v_{g}(S)\leq 1$ for all $g\in G$. Apparently, a squarefree sequence over G can be considered as a subset of $G$. Let $g_0\in G$,  we denote
$$g_0+S=(g_0+g_1)\cdot \ldots \cdot (g_0+g_l).$$

A sequence $S_1\in \mathscr{F}(G)$ is called a subsequence of $S$ if $\mathsf v_{g}(S_1)\leq \mathsf v_{g}(S)$ for all $g\in G$, and denoted by $S_1\mid S$. If $S_1\mid S$, we denote
$$S\cdot S_{1}^{-1}=\prod_{g\in G}g^{\mathsf v_g(S)-\mathsf v_g(S_1)}\in \mathscr{F}(G).$$
If $S_1$ is not a subsequence of $S$, we write $S_1\nmid S$.
Let $S_1, S_2\in \mathscr{F}(G)$, we set
$$S_1\cdot S_2=\prod_{g\in G}g^{\mathsf v_{g}(S_1)+\mathsf v_{g}(S_2)}\in \mathscr{F}(G).$$
Furthermore, we call $S_1,\ldots,S_t\, (t\ge 2)$ are disjoint subsequences of $S$, if $S_1\cdot\ldots\cdot S_t\t S$.

For a sequence
$$S=g_1\cdot \ldots \cdot g_l=\prod_{g\in G}g^{\mathsf v_g(S)}\in \mathscr{F}(G),$$
we list the following definitions
\begin{itemize}
\item[] $|S|=l=\sum_{g\in G}\mathsf v_g(S)\in \mathbb{N}_0$ the \emph{length} of $S$,
\item[] $\mathsf{supp}(S)=\{g\in G\mid \mathsf v_g(S)>0\}\subseteq G$ the \emph{support} of $S$,
\item[] $\sigma(S)=\sum_{i=1}^{l}g_i=\sum_{g\in G}\mathsf v_g(S)g\in G$ the \emph{sum} of $S$,
\item[] $\sum(S)=\{\sum_{i\in I}g_i\mid I\subseteq [1, l] \mbox{ with } 1\leq |I|\leq l\}$ the set of all \emph{subsums} of $S$,

\end{itemize}

The sequence $S$ is called
\begin{itemize}
\item[$\bullet$] \emph{zero-sum free} if $0\notin \sum(S)$,
\item[$\bullet$] a \emph{zero-sum sequence} if $\sigma(S)=0$,
\item[$\bullet$] a \emph{short zero-sum sequence} if $\sigma(S)=0$ and $|S| \in [1, {\exp}(G)]$.

\end{itemize}

Every map of abelian groups $\phi: G\rightarrow H$ extends to a map from the sequences over $G$ to the sequences over $H$ by setting $\phi(S)=\phi(g_1)\cdot \ldots \cdot \phi(g_l)$. If $\phi$ is a homomorphism, then $\phi(S)$ is a zero-sum sequence if and only if $\sigma(S)\in \mathsf{ker}(\phi)$.

Let $G=H\oplus K$ be a finite abelian group. Let $\phi$ denote the projection from  $G$ to  $H$ and $\psi$ denote the projection from $ G$ to  $K$. If $S\in \mathscr{F}(G)$ such that $\sigma(\phi(S))=0$, then $\sigma(S)=\sigma(\psi(S))\in \ker(\phi)=K$.

\begin{lemma} \label{cyclic}

Let $G$ be a cyclic group of order $n\geq 2$.
\begin{enumerate}
\item A sequence $S\in \mathscr{F}(G)$ is zero-sum free of length $|S|=n-1$ if and only if $S=g^{n-1}$ for some $g\in G$ with $\mathsf{ord}(g)=n$.

\smallskip
\item Let  $S$ be a zero-sum free sequence over $G$ with length greater than $n/2$. Then there exists an element $g\in G$ of order $n$  such that  $$S=(k_1g)\cdot\ldots\cdot(k_{|S|}g),$$ where $1\le k_1\le\cdots\le k_{|S|}$, $k=k_1+\cdots+k_{|S|}<n$, and $\Sigma(S)=\{g,2g,\ldots,kg\}$.

\smallskip
\item  Let $S\in \mathscr{F}(G)$ a sequence of length $|S|=\mathsf s(G)-1$. Then the following statements are equivalent:
      \begin{itemize}
      \item[(a)] $S$ has no zero-sum subsequence of length $n$.
      \item[(b)] $S=(gh)^{n-1}$ where $g, h\in G$ with $\mathsf{ord}(g-h)=n$.
      \end{itemize}
\end{enumerate}
\end{lemma}

\begin{proof}
See \cite[Cor. 2.1.4, Th.5.1.8, Prop.5.1.12]{GE}.
\end{proof}

\begin{lemma}(\cite[Theorem 1.2]{FYS})\label{ETAF}
Let $G=H \oplus C_{mn}$ be a finite abelian group where $H \subseteq G$ is a subgroup  with ${\exp}(H)=m\geq 2$ and $n \in \N$. If $n\geq \max\{m|H|+1, 4|H|+2m\}$, then $\mathsf s(G)=\eta(G)+{\exp}(G)-1$.
\end{lemma}

\begin{lemma} \label{ETA}
Let $G=H\oplus C_n$ be a finite abelian group  where $H\subseteq G$
is a subgroup with $\exp (H) \mid \exp (G)  = n$. Then
\[
\eta (G) \ge 2(\mathsf D(H)-1) + n \quad \text{ and} \quad  \mathrm \mathsf s(G) \ge 2(\mathsf D(H)-1) + 2n-1 \,.
\]
In particular, we have for all $ r, n\in \mathbb{N}$,
\[
\eta(C_2^{r-1}\oplus C_{2n})\geq 2n+2r-2 \quad \text{and} \quad \mathsf s(C_2^{r-1}\oplus C_{2n})\geq 4n+2r-3 \,.
\]
\end{lemma}

\begin{proof}
The main statement follows from \cite[Lemma 3.2]{Edel2007}. If $H = C_2^{r-1}$, then $\mathsf D (H)=r$ and the second statement follows.
\end{proof}

\begin{lemma}\label{SUM} Let $W$ be a squarefree sequence of length $|W|\geq 9$ over $C_2^4\setminus \{0\}$ . Then for any element $w\t W$,  there exist $|W|-8$ disjoint subsequences $R_1, \ldots, R_{|W|-8}$ of $Ww^{-1}$ such that $\sigma(R_i)=w$ and $|R_i|=2$ for all $i\in [1, |W|-8]$. In particular, $W$ has at least $|W|-8$ distinct zero-sum subsequences of length $3$ containing $w$.
\end{lemma}

\begin{proof} For any element $w\t W$,  we can always find a set $A=\{a_1,\ldots,a_7\}\subseteq C_2^4\setminus\{0\}$ satisfying  $C_2^4\setminus\{0\}=\{w\}\cup A\cup w+A$. Let $T_i=a_i\cdot(w+a_i)$ for each $i\in [1,7]$.  Therefore $w=\sigma(T_1)=\ldots=\sigma(T_7)$ and $\mathsf{supp}(wT_1\cdot\ldots\cdot T_7)=C_2^4\setminus \{0\}$.

Since \begin{align*}
&\left|\left\{T_i \ \Big|\  i\in [1,7]\text{ and } T_i \nmid W\right\}\right|\le \left|\left\{g\in C_2^4\setminus \{0\}\ \Big|\ g\t T_i \text{ for some }i\in [1,7] \text{ and } g\nmid W \right\}\right|\\
\le& |C_2^4\setminus(\{0\}\cup \supp(W))|= 15-|W|,
\end{align*}
we have that $\left|
\left\{T_i \ \Big|\  i\in [1,7]\text{ and } T_i \t W\right\}\right|\ge 7-(15-|W|)=|W|-8\ge 1$.
In particular, for each  $i\in [1,7]$, if $T_i\t W$, then $wT_i$ is a zero-sum subsequence of $W$.
\end{proof}

\section{The proof of Theorem \ref{Th2} and Theorem \ref{Th1}.1}

\begin{lemma}\label{Supp} Let $G=H\oplus K$ be a finite abelian group, where $H\cong C_2^r$ with $r\geq 3$ a positive integer and $K\cong C_n$ with $n\geq 3$ an odd integer. Denote $\phi_r$ to be the projection from  $G$ to  $H$ and $\psi_r$ to be the projection from $ G$ to  $K$.

 Let $S_r$ be a sequence over $C_2^r\oplus C_n$ such that $\phi_r(S_r)$ is a squarefree sequence  with $\mathsf{supp}(\phi_r(S_r))=H\setminus \{0\}$. If the following property {\bf P1} holds,
 \begin{itemize}
 \item[\bf P1.]  For any two distinct subsequences $T_1, T_2$ of $ S_r$ with $|T_1|=|T_2|=4$ and  $\phi_r(\sigma(T_1))=\phi_r(\sigma(T_2))=0$,  we have that  $\sigma(T_1)=\sigma(T_2)$.
 \end{itemize}
  then  $|\mathsf{supp}(\psi_r(S_r))|=1$.
\end{lemma}

\begin{proof}
 We proceed by induction on $r$.

Suppose that $r=3$. Let $(e_1,e_2,e_3)$ be a basis of $H$, and $e$ be a basis of $K$.

 Since $\phi_3(S_3)$ is a squarefree sequence with $\mathsf{supp}(\phi_3(S_3))=H\setminus \{0\}$, we can assume $S_3=g_1g_2g_3g_4g_5g_6g_7$,
where
\begin{align*}g_1=e_1+a_1e,&\quad g_2=e_2+e_3+a_2e,\quad g_3=e_2+ a_3e, \quad g_4=e_1+e_3+ a_4e ,\\
g_5=e_3+ a_5e ,&\quad g_6=e_1+e_2+a_6e ,\quad  g_7=e_1+e_2+e_3+ a_7e,\text{ and }\ a_1, \ldots , a_7\in [0,n-1].
\end{align*}

By the property {\bf P1} and
\begin{align*}
&\phi_3(g_1+g_3+g_5+g_7)
=\phi_3(g_1+g_4+g_6+g_7)
=\phi_3(g_2+g_3+g_6+g_7)
=\phi_3(g_2+g_4+g_5+g_7)\\
=&\phi_3(g_1+g_2+g_3+g_4)
=\phi_3(g_1+g_2+g_5+g_6)
=\phi_3(g_3+g_4+g_5+g_6)=0,
\end{align*}
  we have that \begin{align*}
  g_1+g_3+g_5+g_7
  &=g_1+g_4+g_6+g_7
  =g_2+g_3+g_6+g_7
  =g_2+g_4+g_5+g_7\\
  &=g_1+g_2+g_3+g_4
  =g_1+g_2+g_5+g_6
  =g_3+g_4+g_5+g_6.
  \end{align*}
By easily calculation,  we obtain that $a_1=a_2=a_3=a_4=a_5=a_6=a_7$, which implies that $|\mathsf{supp}(\psi_3(S_3))|=1$.

\medskip
Suppose that the conclusion is correct for $r=d\geq 3$.
We want to prove the conclusion is also correct for   $r=d+1$.

 Let $(e_1,\dots,e_{d+1})$ be a basis of $H$ and $e$ be a basis of $K$. Denote by $A_i=\left<e_1,\ldots,e_i\right>\setminus \{0\}$, where $i\in[1,d+1]$. Then $A_{d+1}=A_d\cup (e_{d+1}+A_d)\cup \{e_{d+1}\}$. Since $\phi_{d+1}(S_{d+1})$ is a squarefree sequence with  $\mathsf{supp}(\phi_{d+1}(S_{d+1}))=H\setminus \{0\}$, we can assume $S_{d+1}=\prod_{u\in A_{d+1}}u+c_ue$, where $c_u\in [0,n-1]$ for each $u\in A_{d+1}$.

Let
$$W_1=\prod_{u\in A_d}u+c_ue, \quad   W_2=\prod_{u\in e_{d+1}+A_d}u-e_{d+1}+c_ue,\quad  H'=\left<e_1,\ldots,e_d\right>, \text{ and } G'=H'\oplus K .$$
Then $W_1,W_2$ are sequences over $G'$ and $S_{d+1}=W_1\cdot (e_{d+1}+W_2)\cdot (e_{d+1}+c_{e_{d+1}}e)$. Since $\mathsf r(H')=d$, $\phi_{d+1}\big|_{G'}$ is the projection from $G'$ to $H'$, and $\psi_{d+1}\big|_{G'}$ is the projection from $G'$ to $K$, we obtain that
$W_1, W_2$ satisfy the property {\bf P1} for $r=d$ which implies that
 $|\supp(\psi_{d+1}(W_1))|=|\supp(\psi_{d+1}(W_2))|=1$. Therefore we can assume that  $c_u=x$ for all $u\in A_d$  and $c_u=y$ for all $u\in e_{d+1}+A_d$ where $x,y\in [0,n-1]$.

Let $B=\{0,e_1,e_2,e_1+e_2\}$ and  $T_1=\prod_{u\in e_3+ B}u+c_ue$, $T_2=\prod_{u\in e_{d+1}+e_3+B}u+c_ue$, $T_3=\prod_{u\in e_{d+1}+B}u+c_ue$. Then $T_1T_2T_3\t S_{d+1}$ and  $\phi_{d+1}(\sigma(T_1))=\phi_{d+1}(\sigma(T_2))=\phi_{d+1}(\sigma(T_3))=0$. By the property {\bf P1} and $|T_1|=|T_2|=|T_3|=4$, we obtain that  $\psi_{d+1}(\sigma(T_1))=\psi_{d+1}(\sigma(T_2))=\psi_{d+1}(\sigma(T_3))$ and hence  $$4x\equiv 4y\equiv 3y+c_{e_{d+1}}\pmod n.$$ It follows that $x=y=c_{e_{d+1}}$ which implies that $|\supp(\psi(S_{d+1}))|=1$.
\end{proof}

\medskip
\begin{proof}[\bf Proof of Theorem \ref{Th2}] Let $G=H\oplus K$ be a finite abelian group, where $H\cong C_2^r$ with $r\geq 3$ a positive integer and $K\cong C_n$ with $n\geq 3$ an odd integer. Denote $\phi$ to be the projection from  $G$ to  $H$ and $\psi$ to be the projection from $ G$ to  $K$.

Let $S$ be a sequence over $G$ with $|S|=4n+2^r-5$. Assume to the contrary that $S$ contains no zero-sum subsequence of length $2n$.

Suppose that $\phi(G)=H=\{h_0,h_1,\ldots,h_{2^r-1}\}$. We can assume that
$$\phi(S)=h_0^{n_0}\cdot \ldots \cdot h_{2^r-1}^{n_{2^r-1}}\mbox{ and }S=W_0\cdot \ldots \cdot W_{2^r-1}$$
where $n_0, \ldots, n_{2^r-1}\in \mathbb{N}_0$, and $\phi(W_i)=h_{i}^{n_i}$ for all $i\in [0, 2^r-1]$.

Then $S$ allows a product decomposition $S=S_1\cdot \ldots \cdot S_{k}\cdot S_0$  satisfying that  $\phi(S_0)$ is squarefree, and  for each $i\in [1, k]$, $|S_i|=2$ and $\sigma(\phi(S_i))=0$. Therefore $S_i\t W_j$ for some $j\in [0,2^r-1]$.

Since $4n+2^r-5=|S|=2k+|S_0|\le 2k+2^r$, we obtain that $k\ge 2n-2$.
By our assumption, $\sigma(S_1)\cdot\ldots\cdot \sigma(S_k)\in \mathscr F(\ker(\phi))$ has no subsequence of length $n$. Therefore by Lemma  \ref{cyclic}.3, we have that $k=2n-2$, $|S_0|=2^r-1$ and $$\sigma(S_1)\cdot \ldots \cdot \sigma(S_{2n-2})=g^{n-1}{g_1}^{n-1},$$
where $g, g_1\in \mathsf{ker}(\phi)$ and $\mathsf{ord}(g-g_1)=n$.

Since $|\supp(\phi(S_0))|=|S_0|=2^r-1$, let $\{b\}=\phi(G)\setminus\supp(\phi(S_0))$ and $$S'=S+b-\frac{n+1}{2}g_1=(S_1+b-\frac{n+1}{2}g_1)\cdot \ldots \cdot( S_{2n-2}+b-\frac{n+1}{2}g_1)\cdot (S_{0}+b-\frac{n+1}{2}g_1).$$
Therefore $S'$ has no zero-sum subsequence of length $2n$ and $0\notin \supp(\phi(S_0+b))=\supp(\phi(S_0+b-\frac{n+1}{2}g_1))$,
$$\sigma(S_1+b-\frac{n+1}{2}g_1)\cdot \ldots \cdot\sigma( S_{2n-2}+b-\frac{n+1}{2}g_1)=(g-g_1)^{n-1}\cdot 0^{n-1}.$$
Then without loss of generality, we can assume that $0\notin \supp(\phi(S_0))$ and  $$\sigma(S_1)=\cdots =\sigma(S_{n-1})=g_1=0\mbox{ and }\sigma(S_n)=\cdots =\sigma(S_{2n-2})=g.$$

\medskip\noindent
{\bf Case 1.} There exists $T\mid S_0$ such that $|T|=4$,  $\sigma(\phi(T))=0$ and  $\sigma(\psi(T))\neq g$.
\smallskip

Then  $\psi(\sigma(T))=tg$ where $t\in [2, n]$.  By calculation we get
$$\sigma(S_1\cdot\ldots\cdot S_{t-2}\cdot S_{n}\cdot \ldots\cdot S_{2n-1-t}\cdot T)=0, \quad \text{ and }$$
$$|S_1\cdot\ldots\cdot S_{t-2}\cdot S_{n}\cdot \ldots\cdot S_{2n-1-t}\cdot T|
=2(t-2)+2(n-t)+4
=2n,$$
a contradiction.

\medskip\noindent
{\bf Case 2.} For any subsequence $T\t S_0$ satisfying $|T|=4$ and $\phi(\sigma(T))=0$, we have $\psi(\sigma(T))=g$.
\smallskip

By $\mathsf{supp}(\phi(S_0))=\phi(G)\setminus \{0\}$ and Lemma \ref{Supp}, we have $\mathsf{supp}(\psi(S_0))=\{a\}$ for some $a\in \psi(G)$. Let $T$ be a subsequence of $S_0$ satisfying $|T|=4$ and $\phi(\sigma(T))=0$, then $\psi(\sigma(T))=4a=g$ which implies that
\begin{equation}\label{e3.1}
 a=\{\frac{n+1}{4}g\} \quad \text{ if } n\equiv 3\pmod 4 \qquad \text{ and } \qquad a=\{\frac{3n+1}{4}g\} \quad \text{ if } n\equiv 1\pmod 4\,.
\end{equation}

 Without loss of generality, we can  assume $h_0=0\notin \supp(\phi(S_0))$ and
$|W_1|\geq \cdots \geq |W_{2^r-1}|$.
Then $2\nmid |W_i|$ for all $i\in [1, 2^r-1]$.

 We can distinguish the following two cases.

\medskip\noindent
{\bf Subcase 2.1.} $|W_1|\geq 3$.
\smallskip

Without loss of generality, we can assume that  $a_1a_2a_0\t W_1$ and $S_i=a_1a_2$, $a_0\t S_0$, where $i\in [1,2n-2]$. Then $\sigma(\psi(S_i))=\psi(a_1)+\psi(a_2)\in \{0, g\}$ and hence $\psi(a_1)\neq \psi(a_0)$ or $\psi(a_2)\neq \psi(a_0)$ by Equation \eqref{e3.1}. Without loss of generality, we can assume that $\psi(a_1)\neq \psi(a_0)$. Let $S_i'=a_0a_2$ and $S_0'=S_0a_0^{-1}a_1$. Since $\sigma(\phi(S_1))\cdot\ldots \cdot \sigma(S_{i-1})\cdot\sigma(S_i')\cdot \sigma(S_{i+1})\cdot\ldots\cdot \sigma(S_{2n-2})$ has no subsequence of length $n$,  by Lemma \ref{cyclic}.3 we obtain that $\sigma(\phi(S_i'))=\sigma(\phi(S_i))$ which implies that $\psi(a_1)= \psi(a_0)$, a contradiction.

\medskip\noindent
{\bf Subcase 2.2.} $|W_1|\le 2$.
\smallskip

Since $2\geq |W_1|\geq \cdots \geq |W_{2^r-1}|$ and $2\nmid |W_i|$ for all $i\in [1, 2^r-1]$, we have that $|W_1|=\cdots =|W_{2^r-1}|=1$. Then
$$|W_0|=|S|-\sum_{i=1}^{2^r-1}|W_i|=4n+2^r-5-(2^r-1)=4n-4\ge 3n-1.$$
Since $W_0$ is a sequence over $\ker(\phi)\cong C_n$, there are two disjoint  zero-sum subsequences $V_1,V_2$ of length $n$ by $\mathrm \mathsf s(C_n)=2n-1$(see Theorem {\bf A}). Therefore $V_1V_2$ is a zero-sum subsequence of length $2n$, a contradiction.

\end{proof}

\medskip
\begin{proof}[\bf{Proof of Theorem \ref{Th1}.1}]
 By Lemma \ref{ETA} and Inequality \ref{LOWETA}, we only need to prove that $\mathrm \mathsf s(G)\le 4n+3$. If $n$ is odd, it follows immediately by Theorem \ref{Th2}. Thus we can always assume that $n$ is even.

  Let $(e_1,e_{2},e)$ be a basis of $G=C_2^{2}\oplus C_{2n}$ with $\ord(e_1)=\ord(e_{2})=2$ and $\ord(e)=2n$  and $S$ be any sequence over $G$ with $|S|=4n+3$.  Assume to the contrary that $S$ contains no zero-sum subsequence of length $2n$.

Let $\theta: G\rightarrow G$ be the homomorphism defined by $\theta(e_1)=e_1$, $\theta(e_2)=e_2$  and $\theta(e)=ne$. Then $\mathsf{ker}(\theta)=\left<2e\right>\cong C_n$ and $\theta(G)=\left<e_1,e_{2}, ne\right>\cong C_2^3$.

Let $\theta(G)=\{h_0,h_1,\ldots,h_{7}\}$. We can assume that
$$\theta(S)=h_0^{n_0}\cdot \ldots \cdot h_{7}^{n_{7}}\mbox{ and }S=W_0\cdot \ldots \cdot W_{7}$$
where $n_0, \ldots, n_{7}\in \mathbb{N}_0$, and $\theta(W_i)=h_{i}^{n_i}$ for all $i\in [0, 7]$.

Then $S$ allows a product decomposition $S=S_1\cdot \ldots \cdot S_{k}\cdot S_0$  satisfying that  $\theta(S_0)$ is squarefree, and  for each $i\in [1, k]$, $|S_i|=2$ and $\sigma(\theta(S_i))=0$. Therefore $S_i\t W_j$ for some $j\in [0,7]$.

Since $4n+3=|S|=2k+|S_0|\le 2k+8$, we obtain that $k\ge 2n-2$.
By our assumption, $\sigma(S_1)\cdot\ldots\cdot \sigma(S_k)\in \mathscr F(\ker(\theta))$ has no subsequence of length $n$. Therefore by Lemma  \ref{cyclic}.3, we have that $k=2n-2$, $|S_0|=7$ and $$\sigma(S_1)\cdot \ldots \cdot \sigma(S_{2n-2})=(2ke)^{n-1}(2k_1e)^{n-1},$$
where $k, k_1\in [0,n-1]$ and $\gcd(k-k_1,n)=1$. Since $n$ is even, without loss of generality, we can assume that $k_1$ is even.

Since $|\supp(\theta(S_0))|=|S_0|=7$, we let  $\{b\}=\theta(G)\setminus\supp(\theta(S_0))\in G$ and let $$S'=S+b-k_1e=(S_1+b-k_1e)\cdot \ldots \cdot( S_{2n-2}+b-k_1e)\cdot (S_{0}+b-k_1e).$$
Therefore $S'$ has no zero-sum subsequence of length $2n$ and $0\notin \supp(\theta(S_0+b))=\supp(\theta(S_0+b-k_1e))$,
$$\sigma(S_1+b-k_1e)\cdot \ldots \cdot\sigma( S_{2n-2}+b-k_1e)=(2(k-k_1)e)^{n-1}\cdot 0^{n-1}.$$
Then without loss of generality, we can assume that $0\notin \supp(\theta(S_0))$ and  $$\sigma(S_1)=\cdots =\sigma(S_{n-1})=2k_1e=0\ \mbox{ and }\ \sigma(S_n)=\cdots =\sigma(S_{2n-2})=2ke=2e.$$

\medskip\noindent
{\bf Case 1.} There exists $T\mid S_0$ such that $|T|=4$,  $\theta(\sigma(T))=0$ and  $\sigma(T)\neq 2e$.
\smallskip

Then  $\sigma(T)=2te$ where $t\in [2, n]$.  By calculation we get
$$\sigma(S_1\cdot\ldots\cdot S_{t-2}\cdot S_{n}\cdot \ldots\cdot S_{2n-1-t}\cdot T)=0 \quad \text{ and }$$
$$|S_1\cdot\ldots\cdot S_{t-2}\cdot S_{n}\cdot \ldots\cdot S_{2n-1-t}\cdot T|
=2(t-2)+2(n-t)+4
=2n,$$
a contradiction.

\medskip\noindent
{\bf Case 2.} For any subsequence $T\t S_0$ satisfying $|T|=4$ and $\theta(\sigma(T))=0$, we have $\sigma(T)=2e$.
\smallskip

 Since $\mathsf{supp}(\theta(S_0))=\theta(G)\setminus \{0\}$, we can assume $S_0=g_1g_2g_3g_4g_5g_6g_7$,
where
\begin{align*}g_1=e_1+2a_1e,&\quad g_2=e_2+e+2a_2e,\quad g_3=e_2+ 2a_3e, \quad g_4=e_1+e+ 2a_4e ,\\
g_5=e+ 2a_5e ,&\quad g_6=e_1+e_2+2a_6e ,\quad  g_7=e_1+e_2+e+ 2a_7e,\text{ and }a_1, \ldots , a_7\in [0,n-1].
\end{align*}

Since $
\theta(g_1+g_3+g_5+g_7)
=\theta(g_2+g_4+g_5+g_7)
=\theta(g_1+g_2+g_3+g_4)
$,
 we obtain the following equations
\begin{align*}
g_1+g_3+g_5+g_7
=g_2+g_4+g_5+g_7
=g_1+g_2+g_3+g_4
=2e,
\end{align*}

which implies that
$$2e+2(a_1+a_3+a_5+a_7)e=4e+2(a_2+a_4+a_5+a_7)e=2e+2(a_1+a_2+a_3+a_4)e=2e.$$

Therefore \begin{align*}
& a_1+a_3+a_5+a_7\equiv 0 \pmod n,\\
& a_2+a_4+a_5+a_7 \equiv -1 \pmod n, \\
& a_1+a_2+a_3+a_4\equiv 0 \pmod n.
\end{align*}
Thus $2(a_1+a_3)+a_2+a_4+a_5+a_7\equiv 0\pmod n$, which implies that $2(a_1+a_3)\equiv 1\pmod n$, a contradiction to  $n$ is even.
\end{proof}

\section{Preparatory results about $C_2^3\oplus C_{2n}$}

In the whole section, we consider the group $G=C_2^3\oplus C_{2n}$, where $n\ge 3$ is an odd integer. Thus $G\cong C_2^4\oplus C_n$.
Let $G=H\oplus K$, where $H,K$ are subgroups of $G$ with $H\cong C_2^4$  and  $K\cong C_n$. Denote $\phi$ to be the projection from  $G$ to  $H$ and $\psi$ to be the projection from $ G$ to  $K$.

\begin{lemma}\label{IMP} Let $G,H,K$ and $\phi,\psi$ be as above.
If  $S$ is a sequence of length $10$ over $G$ such that  $\phi(S)$  is a squarefree sequence over $H\setminus\{0\}$,  then  $S$  has two distinct  subsequences $T_1$ and $T_2$ of length $\{ |T_1|, |T_2|\}
 \subseteq [3,4]$ satisfying $\sigma(\phi(T_1))=\sigma(\phi(T_2))=0$ but $\sigma(\psi(T_1))\neq \sigma(\psi(T_2))$.
\end{lemma}

\begin{proof}
By Lemma \ref{SUM}, $S$ has at least two distinct subsequences $W_1$, $W_2$ of length $3$ such that $\phi(\sigma(W_1))=\phi(\sigma(W_2))=0$.

Assume to the contrary that  for any  zero-sum subsequence $\phi(T)$ of $\phi(S)$ with length $3$ or $4$, we have  $\sigma(\psi(T))=e$, where $e\in K\setminus \{0\}$.

By Lemma \ref{SUM}, there exists a subsequence $T$ of $S$ such that $|T|=3$ and  $\sigma(\phi(T))=0$.
Then $\sigma(\psi(T))=e$ and hence there exists an element $u\t T$ such that $\psi(u)\neq \frac{n+1}{2}e$.

By Lemma \ref{SUM} again, there exist disjoint subsequences $R_1, \, R_2$  of  $Su^{-1}$ such that $\sigma(\phi(R_1))=\sigma(\phi(R_2))=\phi(u)$ and $|R_1|=|R_2|=2$.
Thus $\phi(R_1R_2)$, $\phi(R_1u)$, and $\phi(R_2u)$ are zero-sum sequences which implies that $\sigma(\psi(R_1R_2))=\sigma(\psi(R_1u))=\sigma(\psi(R_2u))=e$.
It follows that $\psi(u)=\frac{n+1}{2}e$, a contradiction to the choice of $u$.
\end{proof}

\begin{lemma}\label{SHO} Let $G,H,K$ and $\phi,\psi$ be as above.
 If $n=3$ and  $S$ is a sequence of length $12$ over $G$ such that  $\phi(S)$ is a  squarefree sequence,  then  $S$ contains a short zero-sum subsequence.
\end{lemma}

\begin{proof}

 Assume to the contrary that $S$ contains no short zero-sum subsequence. Thus $0\not\in \mathsf{supp}(S)$.

If $0\in \mathsf{supp}(\phi(S))$, then there exists $g\mid S$ such that $\phi(g)=0$ and hence $\psi(g)\neq 0$. By $|Sg^{-1}|=11$ and  Lemma \ref{IMP}, $Sg^{-1}$  have a subsequence $T$ of length $|T|\in\{3,4\}$ such that $\sigma(\phi(T))=0$ and $\sigma(\psi(T))\neq \psi(g)$. Since $\sigma(\psi(T))\neq 0$, we obtain $\sigma(\psi(T))=2\psi(g)$ which implies that $Tg$  is a short zero-sum subsequence of $S$, a contradiction.

\medskip
Therefore  $0\notin \mathsf{supp}(\phi(S))$. Let $S=g_1\cdot\ldots\cdot g_{12}$.  We distinguish the following four cases to finish the proof.

\medskip
\noindent\textbf{Case 1.} $\mathsf v_0(\psi(S))\geq 5$.
\smallskip

Without loss of generality, we can assume that $\psi(g_1\cdot\ldots\cdot g_5)=0^5$. Since $\phi(g_1\cdot\ldots\cdot g_5)\in \mathscr{F}(\phi(G))$ and $\mathsf D(\phi(G))=\mathsf D(C_2^4)=5$, there exists a subsequence $X\mid g_1\cdot\ldots\cdot g_5$ such that $\sigma(\phi(X))=0$ which implies that $X$ is a short zero-sum subsequence of $S$, a contradiction.

\medskip
\noindent\textbf{Case 2.} $\mathsf v_0(\psi(S))=4$.
\smallskip

Without loss of generality, we can assume that $\psi(g_1g_2g_3g_4)=0^4$ and $\psi(g_5g_6g_7g_8)=e^4$ for some $e\in \psi(G)\cong C_3$. Thus $g_1, g_2, g_3, g_4, g_5+g_6+g_7, g_5+g_6+g_8\in \phi(G)\cong C_2^4$ and $g_5+g_6+g_7\neq g_5+g_6+g_8$.  Choose $R=g_5+g_6+g_7$ or $g_5+g_6+g_8$ such that $\sigma(R)\neq g_1+g_2+g_3+g_4$. Then $g_1g_2g_3g_4\sigma(R)$ has a zero-sum subsequence of length $\le 4$ which implies that  $S$ contains a short zero-sum subsequence, a contradiction.

\medskip
\noindent\textbf{Case 3.} $\mathsf v_0(\psi(S))=3$.
\smallskip

Without loss of generality, we can assume that  $\psi(S)=0^3\cdot e^u\cdot (2e)^v$, $u+v=9$, $u\geq v$, and $\psi(g_1g_2g_3)=0^3$ for some $e\in \psi(G)\cong C_3$.

Suppose that $v=0$. We assume that $\psi(g_4\cdot \ldots\cdot g_{12})=e^9$. Then by Lemma \ref{SUM}, $\phi(g_4\cdot \ldots\cdot g_{12})$ contains a zero-sum subsequence of length $3$ which implies that $S$ contains a short zero-sum subsequence of length $3$, a contradiction.

Suppose that $v=1$. We assume that $\psi(g_4\cdot \ldots\cdot g_{11})=e^8$ and $\psi(g_{12})=2e$. Then $g_1,g_2,g_3, g_{11}+g_{12}, g_4+g_5+g_j\in \phi(G)\cong C_2^4$ for any $j\in [6,10]$. If $\sigma(g_1g_2g_3g_{11}g_{12})=0$, then $g_1g_2g_3g_{11}g_{12}$ is a short zero-sum subsequence of $S$, a contradiction. Thus $\sigma(g_1g_2g_3g_{11}g_{12})\neq 0$.  Since $|[6,10]|=5$,  there exists an $i\in [6, 10]$ such that $$g_4+g_5+g_i\notin \{\sigma(g_1g_2g_3(g_{11}g_{12})), \sigma(g_1g_2(g_{11}g_{12})), \sigma(g_1g_3(g_{11}g_{12})), \sigma(g_2g_3(g_{11}g_{12}))\}.$$
Therefore $g_1g_2g_3\sigma(g_{11}g_{12})\sigma(g_4g_5g_i)$ contains a zero-sum subsequence of length $\le 3$ which implies that  $S$ contains a short zero-sum subsequence, a contradiction.

Suppose that  $v\geq 2$.  we assume that $\psi(g_4\cdot \ldots\cdot g_{8})=e^{5}$ and $\psi(g_{11} g_{12})=(2e)^2$. Then $g_4+g_{11}, g_5+g_{12}, g_6+g_{12}\in \phi(G)\cong C_2^4$ and $g_5+g_{12}\neq g_6+g_{12}$. Choose $R=g_5+g_{12}$ or $g_6+g_{12}$ such that  $\sigma(R)\neq g_{4}+g_{11}+g_1+g_2+g_3$. Therefore $g_1g_2g_3\sigma(g_{4}g_{11})\sigma(R)$ contains a zero-sum subsequence of length $\le4$ which implies that $S$ contains a short zero-sum subsequence, a contradiction.

\medskip
\noindent\textbf{Case 4.} $\mathsf v_0(\psi(S))\le 2$.
\smallskip

Without loss of generality, we can assume that  $\psi(S)=0^t\cdot e^u\cdot (2e)^v$, $t\in \mathbb{N}_0$, $u+v\ge 10$,  and  $u\geq v$ for some $e\in \psi(G)\cong C_3$.

Suppose that $u\ge 9 $.  We assume that $\psi(g_1\cdot\ldots\cdot g_9)=e^9$. Then by Lemma \ref{SUM}, $\phi(g_1\cdot \ldots\cdot g_{9})$ contains a zero-sum subsequence of length $3$ which implies that $S$ contains a short zero-sum subsequence of length $3$, a contradiction.

Suppose that $u\le 8$. Then $v\ge 2$. We assume that $\psi(g_1\cdot \ldots\cdot g_{u})=e^{u}$ and $\psi(g_{u+1}\cdot \ldots\cdot g_{u+v})=(2e)^v$. If there exist $i_1, i_2\in [1, u]$ and $j_1, j_2\in [u+1, u+v]$ such that $i_1\neq i_2$, $j_1\neq j_2$, and $g_{i_1}+g_{j_1}=g_{i_2}+g_{j_2}$, then $g_{i_1}g_{j_1}g_{i_2}g_{j_2}$ is a short zero-sum subsequence of $S$, a contradiction. Therefore
$$|\{g_i+g_j\in \phi(G)\mid i\in [1,u] \text{\ and\ } j\in[u+1,u+v] \}|\ge uv\ge v(10-v)\ge 16=|\phi(G)|.$$
It follows that there exist an $i\in [1,u]$ and a $j\in [u+1,u+v]$ such that $\phi(g_i)=\phi(g_j)$,
a contradiction to $\phi(S)$ is squarefree.
\end{proof}

\begin{lemma}\label{L1} Let $G,H,K$ and $\phi,\psi$ be as above.
Let $ K=\left<e\right>$  and  $S=h_1\cdot\ldots\cdot h_8$ be a sequence over $G\setminus\{0\}$ with $\phi(h_1)+\phi(h_2)=\phi(h_3)+\phi(h_4)=\phi(h_5)+\phi(h_6)=\phi(h_7)+\phi(h_8)$.
If  $\phi(S)$  is a squarefree sequence with $0\not\in \supp(\phi(S))$ and satisfies the following property $(*)$:
  \[
    \left\{\begin{aligned} & \text{For any subsequence $V$ of $S$ with $\sigma(\phi(V))=0$, \text{ we have that }}\\
     &\sigma(\psi(V)) =\left \{ \begin{aligned}&  e ,&\quad \quad \quad& \mbox{ if } |V|=3\mbox{ or }4, \\
    & e\mbox{ or }2e,&\quad &\mbox{ if } |V|=5.
    \end{aligned} \right.
   \end{aligned}\qquad (*)\right.
    \]
 then $\supp(\psi(S))=\{\frac{n+1}{4}e\}$ if $n\equiv 3 \pmod{4}$ and  $\supp(\psi(S))=\{\frac{3n+1}{4}e\}$ if $n\equiv 1 \pmod{4}$.
\end{lemma}

\begin{proof}
Since $\sigma(\phi(h_1h_2h_3h_4))=\sigma(\phi(h_3h_4h_5h_6))=\sigma(\phi(h_5h_6h_1h_2))=0$, we obtain that $\sigma(\psi(h_1h_2h_3h_4))=\sigma(\psi(h_3h_4h_5h_6))=\sigma(\psi(h_5h_6h_1h_2))=e$ which implies that $\psi(h_1)+\psi(h_2)=\psi(h_3)+\psi(h_4)=\psi(h_5)+\psi(h_6)=\frac{n+1}{2}e$. With the same reason, we can prove that $\psi(h_7)+\psi(h_8)=\frac{n+1}{2}e$.

Let $\psi(h_i)=k_ie$ where $1\le i\le 8$ and $0\le k_i\le n-1$. Without loss of generality, we can assume that $k_1\le k_2,\, k_3\le k_4,\, k_5\le k_6,\,k_7\le k_8$. We consider the sequence $W=h_1h_2h_3h_5h_7$ (see Figure \ref{f1}).

\begin{figure}[ht]
\begin{center}
\setlength{\unitlength}{0.7 mm}%
\begin{picture}(76.76,37.24)(0,0)
\put(7.00,27.04){\circle*{1.80}}
\put(6.96,6.69){\circle*{1.80}}
\put(27.31,27.04){\circle*{1.80}}
\put(26.96,6.69){\circle*{1.80}}
\put(46.97,27.04){\circle*{1.80}}
\put(66.97,27.04){\circle*{1.80}}
\put(67.52,6.69){\circle*{1.80}}
\put(3.65,28.85){\fontsize{8.53}{10.24}\selectfont \makebox(7.5, 3.0)[l]{$h_1$\strut}}
\put(3.71,9.67){\fontsize{8.53}{10.24}\selectfont \makebox(7.5, 3.0)[l]{$h_2$\strut}}
\put(23.90,28.85){\fontsize{8.53}{10.24}\selectfont \makebox(7.5, 3.0)[l]{$h_3$\strut}}
\put(23.55,9.67){\fontsize{8.53}{10.24}\selectfont \makebox(7.5, 3.0)[l]{$h_4$\strut}}
\put(43.44,28.85){\fontsize{8.53}{10.24}\selectfont \makebox(7.5, 3.0)[l]{$h_5$\strut}}
\put(44.01,9.67){\fontsize{8.53}{10.24}\selectfont \makebox(7.5, 3.0)[l]{$h_6$\strut}}
\put(63.52,28.85){\fontsize{8.53}{10.24}\selectfont \makebox(7.5, 3.0)[l]{$h_7$\strut}}
\put(64.09,9.67){\fontsize{8.53}{10.24}\selectfont \makebox(7.5, 3.0)[l]{$h_8$\strut}}
\put(47.05,6.69){\circle*{1.80}}
\thinlines\lbezier(2.00,35.24)(74.76,35.24)\lbezier(74.76,35.24)(74.76,22.48)\lbezier(74.76,22.48)(12.19,22.48)\lbezier(12.19,22.48)(12.19,2.00)\lbezier(12.19,2.00)(2.00,2.00)\lbezier(2.00,2.00)(2.00,35.24)
\end{picture}%

\end{center}
\caption{}\label{f1}
\end{figure}

Since $\phi(W)\in \mathscr{F}(\phi(G))$ and $\mathsf D(\phi(G))=\mathsf D(C_2^4)=5$, there exists a subsequence $V\mid W$ such that $\sigma(\phi(V))=0$ and $|V|\in \{3, 4, 5\}$. We distinguish three cases depending on $|V|$.

\smallskip
\noindent\textbf{Case 1.} $|V|=3$.
\smallskip

Since $0\not\in \supp(\phi(S))$, we obtain that $h_1h_2\nmid V$.
By symmetry, we only need to consider $V=h_1h_3h_5$ or $V=h_2h_3h_5$.

Suppose that $V=h_1h_3h_5$. Then $\sigma(\phi(h_1h_3h_5))=0$ and hence $\sigma(\phi(h_1h_4h_6))=0$. Thus \[\sigma(\psi(h_1h_3h_5))=\sigma(\psi(h_1h_4h_6))=e\]
 which implies that $\sigma(\psi(h_3h_5))=\sigma(\psi(h_4h_6))=\frac{n+1}{2}e$ and $\psi(h_1)=\frac{n+1}{2}e$ by $\sigma(\psi(h_3h_4h_5h_6))=e$. Therefore $\psi(h_2)=0$, a contradiction to $k_1\le k_2$.

Suppose that $V=h_2h_3h_5$. Then $\sigma(\phi(h_2h_3h_5))=0$ and hence $\sigma(\phi(h_1h_4h_5))=0$. Thus \[\sigma(\psi(h_2h_3h_5))=\sigma(\psi(h_1h_4h_5))=e\] which implies that $\sigma(\psi(h_2h_3))=\sigma(\psi(h_1h_4))=\frac{n+1}{2}e$ and $\psi(h_5)=\frac{n+1}{2}e$ by $\sigma(\psi(h_1h_2h_3h_4))=e$. Therefore $\psi(h_6)=0$, a contradiction to $k_5\le k_6$.

\smallskip
\noindent\textbf{Case 2.} $|V|=4$.
\smallskip

Since $\phi(S)$ is squarefree, we obtain that $h_1h_2\nmid V$.
Thus there are only two cases: $V=h_1h_3h_5h_7$ and $V=h_2h_3h_5h_7$.
\medskip

Suppose that $V=h_1h_3h_5h_7$.  Since $\sigma(\phi(h_1h_3h_5h_7))=\sigma(\phi(h_2h_4h_5h_7))=0$, we have that  \[\sigma(\psi(h_1h_3h_5h_7))=\sigma(\psi(h_2h_4h_5h_7))=e\] which implies that $\psi(h_1+h_3)=\psi(h_2+h_4)$. By $\psi(h_1+h_2)=\psi(h_3+h_4)=\frac{n+1}{2}e$ and $k_1\le k_2,\, k_3\le k_4$, we obtain that  \begin{align*}
\psi(h_1)&=\psi(h_3)=\psi(h_2)=\psi(h_4)=\frac{n+1}{4}e \,\quad \text{ if } n\equiv 3 \pmod{4}\,,\\ \psi(h_1)&=\psi(h_3)=\psi(h_2)=\psi(h_4)=\frac{3n+1}{4}e \quad \text{ if }  n\equiv 1 \pmod{4}\,.
\end{align*}
 With the same reason, we can prove that
 \begin{align*}
 \psi(h_5)&=\psi(h_6)=\psi(h_7)=\psi(h_8)=\frac{n+1}{4}e \,\quad \text{ if } n\equiv 3 \pmod{4}\,, \\ \psi(h_5)&=\psi(h_6)=\psi(h_7)=\psi(h_8)=\frac{3n+1}{4}e \quad \text{ if } n\equiv 1 \pmod{4}\,.
 \end{align*}
Therefore  $\supp(\psi(S))=\{\frac{n+1}{4}e\}$ if $n\equiv 3 \pmod{4}$ and  $\supp(\psi(S))=\{\frac{3n+1}{4}e\}$ if $n\equiv 1 \pmod{4}$.
\medskip

Suppose that $V=h_2h_3h_5h_7$. Since $\sigma(\phi(h_2h_3h_5h_7))=\sigma(\phi(h_2h_3h_6h_8))=0$, we have that \[\sigma(\psi(h_2h_3h_5h_7))=\sigma(\psi(h_2h_3h_6h_8))=e\] which implies that $\sigma(\psi(h_5+h_7))=\sigma(\psi(h_6+h_8))$. By $\psi(h_5+h_6)=\psi(h_7+h_8)=\frac{n+1}{2}e$ and $k_5\le k_6,\, k_7\le k_8$, we obtain that
\begin{align*}
\psi(h_5)&=\psi(h_6)=\psi(h_7)=\psi(h_8)=\frac{n+1}{4}e \quad \text{ if  } n\equiv 3 \pmod{4}\,,\\ \psi(h_5)&=\psi(h_6)=\psi(h_7)=\psi(h_8)=\frac{3n+1}{4}e\quad \text{ if } n\equiv 1 \pmod{4}\,.
\end{align*}
With the same reason, we can prove that
 \begin{align*}
 \psi(h_3)&=\psi(h_4)=\psi(h_5)=\psi(h_6)=\frac{n+1}{4}e \,\quad \text{ if } n\equiv 3 \pmod{4}\,, \\ \psi(h_3)&=\psi(h_4)=\psi(h_5)=\psi(h_6)=\frac{3n+1}{4}e \quad \text{ if } n\equiv 1 \pmod{4}\,.
 \end{align*}
Thus $\sigma(\psi(V))=\sigma(\psi(h_2h_3h_5h_7))=e$ implies that $\psi(h_2)=\psi(h_3)=\psi(h_5)=\psi(h_7)$ and hence  $\psi(h_1+h_2)=\frac{n+1}{2}e$ implies that $\psi(h_1)=\psi(h_2)$.
Therefore  $\supp(\psi(S))=\{\frac{n+1}{4}e\}$ if $n\equiv 3 \pmod{4}$ and  $\supp(\psi(S))=\{\frac{3n+1}{4}e\}$ if $n\equiv 1 \pmod{4}$.

\smallskip
\noindent\textbf{Case 3.} $|V|=5$. Then $V=h_1h_2h_3h_5h_7$.
\smallskip

It follows that $\sigma(\phi(h_4h_6h_8))=0$ and hence $\sigma (\phi(h_4h_5h_7))=\sigma(\phi(h_3h_6h_7))=0$. Thus by Property (*), we obtain $\sigma(\psi(h_4h_5h_7))=\sigma(\psi(h_3h_6h_7))=e$ which implies that $\sigma(\psi(h_4h_5))=\sigma(\psi(h_3h_6))=\frac{n+1}{2}e$ and $\psi(h_7)=\frac{n+1}{2}e$ by $\sigma(\psi(h_3h_4h_5h_6))=e$. Therefore $\psi(h_8)=0$, a contradiction to $k_7\le k_8$.
\end{proof}

\begin{lemma}\label{L2} Let $G,H,K$ and $\phi,\psi$ be as above.
Let $ K=\left<e\right>$ and $S=h_1\cdot\ldots\cdot h_8$ be a sequence over $G\setminus\{0\}$ with $\phi(h_1)=\phi(h_2)+\phi(h_3)=\phi(h_4)+\phi(h_5)=\phi(h_6)+\phi(h_7)$.
If  $\phi(S)$  is a squarefree sequence with $0\not\in \supp(\phi(S))$, then the following property $(*)$ does not hold.
  \[
      \left\{\begin{aligned} & \text{For any subsequence $V$ of $S$ with $\sigma(\phi(V))=0$, \text{ we have that }}\\
       &\sigma(\psi(V)) =\left \{ \begin{aligned}&  e ,&\quad \quad \quad& \mbox{ if } |V|=3\mbox{ or }4, \\
      & e\mbox{ or }2e,&\quad &\mbox{ if } |V|=5.
      \end{aligned} \right.
     \end{aligned}\qquad (*)\right.
      \]
\end{lemma}

\begin{proof}
Assume to the contrary that the property $(*)$ holds.

Since $\sigma(\phi(h_1h_2h_3))=\sigma(\phi(h_2h_3h_4h_5))=\sigma(\phi(h_4h_5h_1))=\sigma(\phi(h_4h_5h_6h_7))=\sigma(\phi(h_6h_7h_1))=0$, we obtain that $\sigma(\psi(h_1h_2h_3))=\sigma(\psi(h_2h_3h_4h_5))=\sigma(\psi(h_4h_5h_1))=\sigma(\psi(h_4h_5h_6h_7)))=\sigma(\psi(h_6h_7h_1))=e$ which implies that $\psi(h_1)=\psi(h_2)+\psi(h_3)=\psi(h_4)+\psi(h_5)=\psi(h_6)+\psi(h_7)=\frac{n+1}{2}e$.

Let $\psi(h_i)=k_ie$ where $1\le i\le 8$ and $0\le k_i\le n-1$. Without loss of generality, we can assume that $k_2\le k_3,\, k_4\le k_5,\, k_6\le k_7$. We consider the sequence $W=h_1h_2h_4h_6h_8$ (see Figure \ref{f2}).

\begin{figure}[ht]
\begin{center}

\setlength{\unitlength}{0.7 mm}%
\begin{picture}(102.66,34.98)(0,0)
\put(30.73,22.97){\circle*{1.80}}
\put(30.39,2.62){\circle*{1.80}}
\put(50.74,22.97){\circle*{1.80}}
\put(50.39,2.62){\circle*{1.80}}
\put(70.40,22.97){\circle*{1.80}}
\put(90.40,22.97){\circle*{1.80}}
\put(10.38,12.45){\circle*{1.80}}
\put(7.20,14.90){\fontsize{8.53}{10.24}\selectfont \makebox(7.5, 3.0)[l]{$h_1$\strut}}
\put(26.76,24.75){\fontsize{8.53}{10.24}\selectfont \makebox(7.5, 3.0)[l]{$h_2$\strut}}
\put(26.90,5.09){\fontsize{8.53}{10.24}\selectfont \makebox(7.5, 3.0)[l]{$h_3$\strut}}
\put(46.60,24.75){\fontsize{8.53}{10.24}\selectfont \makebox(7.5, 3.0)[l]{$h_4$\strut}}
\put(46.87,5.09){\fontsize{8.53}{10.24}\selectfont \makebox(7.5, 3.0)[l]{$h_5$\strut}}
\put(65.91,24.75){\fontsize{8.53}{10.24}\selectfont \makebox(7.5, 3.0)[l]{$h_6$\strut}}
\put(67.13,5.09){\fontsize{8.53}{10.24}\selectfont \makebox(7.5, 3.0)[l]{$h_7$\strut}}
\put(86.49,24.75){\fontsize{8.53}{10.24}\selectfont \makebox(7.5, 3.0)[l]{$h_8$\strut}}
\put(70.47,2.62){\circle*{1.80}}
\thinlines\lbezier(20.76,32.78)(100.47,32.78)\lbezier(100.47,32.78)(100.47,18.12)\lbezier(100.47,18.12)(22.50,18.12)\lbezier(22.50,18.12)(10.48,5.64)\lbezier(10.48,5.64)(2.00,13.93)\lbezier(2.00,13.93)(20.76,32.78)
\end{picture}%

\end{center}
\caption{}\label{f2}
\end{figure}

Since $\phi(W)\in \mathscr{F}(\phi(G))$ and $\mathsf D(\phi(G))=\mathsf D(C_2^4)=5$, there exists a subsequence $V\mid W$ such that $\sigma(\phi(V))=0$ and $|V|\in \{3, 4, 5\}$. We distinguish three cases depending on $|V|$.

\smallskip
\noindent\textbf{Case 1.} $|V|=3$.
\smallskip

Obviously $h_1\nmid v$. Then by symmetry, we only need to consider $V=h_2h_4h_6$ or $h=h_2h_4h_8$.

Suppose that $V=h_2h_4h_6$. Then $\sigma(\phi(h_2h_4h_6))=\sigma(\phi(h_3h_5h_6))=0$ and  hence $\sigma(\psi(h_2h_4h_6))=\sigma(\psi(h_3h_5h_6))=e$. Therefore $\sigma(\psi(h_2h_4))=\sigma(\psi(h_3h_5))=\frac{n+1}{2}e$ by $\sigma(\psi(h_2h_3h_4h_5))=e$. It follows by that  $\psi(h_6)=\frac{n+1}{2}e$, a contradiction to $k_6\le k_7$ and $\psi(h_6)+\psi(h_7)=\frac{n+1}{2}e$.

Suppose that $V=h_2h_4h_8$. Then $\sigma(\phi(h_2h_4h_8))=\sigma(\phi(h_3h_5h_8))=0$ and  hence $\sigma(\psi(h_2h_4h_8))=\sigma(\psi(h_3h_5h_8))=e$. Therefore $\sigma(\psi(h_2h_4))=\sigma(\psi(h_3h_5))=\frac{n+1}{2}e$ by $\sigma(\psi(h_2h_3))=\sigma(\psi(h_4h_5))=\frac{n+1}{2}e$. It follows by $k_2\le k_3$ and $k_4\le k_5$ that $\psi(h_2)=\psi(h_3)=\psi(h_4)=\psi(h_5)$ and  $\psi(h_8)=\frac{n+1}{2}e$. Therefore by $\sigma(\phi(h_1h_3h_4h_8))=\sigma(\phi(h_2h_4h_8))=0$, we obtain that $\sigma(\psi(h_1h_3h_4h_8))=e$. But $\sigma(\psi(h_1h_3h_4h_8))=\frac{n+1}{2}e+\frac{n+1}{2}e+\frac{n+1}{2}e\neq e$, a contradiction.

\smallskip
\noindent\textbf{Case 2.} $|V|=4$.
\smallskip

By symmetry,
we only need to consider $V=h_1h_2h_4h_6$ or $V=h_2h_4h_6h_8$ or $V=h_1h_2h_4h_8$.

Suppose that $V=h_1h_2h_4h_6$. Then $\sigma(\phi(h_1h_2h_4h_6))=\sigma(\phi(h_3h_4h_6))=\sigma(\phi(h_2h_4h_7))=0$. Thus we obtain that  $\sigma(\psi(h_3h_4h_6))=\sigma(\psi(h_2h_4h_7))=e$ which implies that $\sigma(\psi(h_3h_6))=\sigma(\psi(h_2h_7))=\frac{n+1}{2}e$ by $\sigma(\psi(h_2h_3h_6h_7))=e$. Therefore $\psi(h_4)=\frac{n+1}{2}e$, a contradiction to $k_4\le k_5$ and $\psi(h_4+h_5)=\frac{n+1}{2}e$.

Suppose that $V=h_2h_4h_6h_8$. Then $\sigma(\phi(h_2h_4h_6h_8))=\sigma(\phi(h_3h_5h_6h_8))=0$ and hence $\sigma(\psi(h_2h_4h_6h_8))=\sigma(\psi(h_3h_5h_6h_8))=e$. Thus $\sigma(\psi(h_2h_4))=\sigma(\psi(h_3h_5))=\sigma(\psi(h_6h_8))=\frac{n+1}{2}e$. By $k_2\le k_3,k_4\le k_5$, we obtain that  $\psi(h_2)=\psi(h_3)=\psi(h_4)=\psi(h_5)$.  Since $\sigma(\phi(h_1h_3h_4h_6h_8))=\sigma(\phi(h_2h_4h_6h_8))=0$, we obtain that $\sigma(\psi(h_1h_3h_4h_6h_8))=\psi(h_1)+e\in\{e,2e\}$. Thus $\psi(h_1)\in \{0,e\}$, a contradiction to $\psi(h_1)=\frac{n+1}{2}e$.

Suppose that $V=h_1h_2h_4h_8$. Then $\sigma(\phi(h_1h_2h_4h_8))=\sigma(\phi(h_1h_3h_5h_8))=0$ and hence  $\sigma(\psi(h_1h_2h_4h_8))=\sigma(\psi(h_1h_3h_5h_8))=e$. Thus $\sigma(\psi(h_2h_4))=\sigma(\psi(h_3h_5))=\sigma(\psi(h_1h_8))=\frac{n+1}{2}e$. By $k_2\le k_3,k_4\le k_5$, we obtain that  $\psi(h_2)=\psi(h_3)=\psi(h_4)=\psi(h_5)$. Since $\psi(h_1)=\frac{n+1}{2}e$, we obtain that $\psi(h_8)=0$. It follows that $\sigma(\phi(h_3h_4h_8))=\sigma(\phi(h_1h_2h_4h_8))=0$ and $\sigma(\psi(h_3h_4h_8))=\frac{n+1}{2}e+0\neq e$, a contradiction.

\smallskip
\noindent\textbf{Case 3.} $|V|=5$. Then $V=h_1h_2h_4h_6h_8$.
\smallskip

Then $\sigma(\phi(h_1h_2h_4h_6h_8))=\sigma(\phi(h_2h_4h_7h_8))=\sigma(\phi(h_3h_5h_7h_8))=0$ and hence $\sigma(\psi(h_2h_4h_7h_8))=\sigma(\psi(h_3h_5h_7h_8))=e$. Thus $\sigma(\phi(h_2h_4))=\sigma(\phi(h_3h_5))=\frac{n+1}{2}e$.  By $k_2\le k_3,k_4\le k_5$, we obtain that  $\psi(h_2)=\psi(h_3)=\psi(h_4)=\psi(h_5)$. Since $\sigma(\phi(h_1h_2h_4h_6h_8))=\sigma(\phi(h_3h_4h_6h_8))=0$, we obtain that  $\sigma(\psi(h_1h_2h_4h_6h_8))=\psi(h_1)+\sigma(\psi(h_3h_4h_6h_8))=\frac{n+1}{2}e+e\notin\{e,2e\}$, a contradiction.

\end{proof}

\begin{lemma}\label{L3} Let $G,H,K$ and $\phi,\psi$ be as above.
Let $ K=\left<e\right>$ and $S=h_1\cdot\ldots\cdot h_8$ be a sequence over $G\setminus\{0\}$ with $\phi(h_1)=\phi(h_2)+\phi(h_3)=\phi(h_4)+\phi(h_5)$ and $\psi(h_3)=\psi(h_5)=\frac{n+1}{2}e$.
If  $\phi(S)$  is a squarefree sequence with $0\not\in \supp(\phi(S))$, then the following property $(*)$ does not hold.
 \[
     \left\{\begin{aligned} & \text{For any subsequence $V$ of $S$ with $\sigma(\phi(V))=0$, \text{ we have that }}\\
      &\sigma(\psi(V)) =\left \{ \begin{aligned}&  e ,&\quad \quad \quad& \mbox{ if } |V|=3\mbox{ or }4, \\
     & e\mbox{ or }2e,&\quad &\mbox{ if } |V|=5.
     \end{aligned} \right.
    \end{aligned}\qquad (*)\right.
     \]
\end{lemma}

\begin{proof}
Assume to the contrary that the property $(*)$ holds.

Since  $\sigma(\phi(h_1))=\sigma(\phi(h_2h_3))=\sigma(\phi(h_4h_5))$, we obtain that $\sigma(\phi(h_1h_2h_3))=\sigma(\phi(h_2h_3h_4h_5))=\sigma(\phi(h_4h_5h_1))=0$ which implies that $\sigma(\psi(h_1h_2h_3))=\sigma(\psi(h_2h_3h_4h_5))=\sigma(\psi(h_4h_5h_1))=e$. Therefore $\psi(h_1)=\frac{n+1}{2}e$ and $\psi(h_2)=\psi(h_4)=0$ by $\psi(h_3)=\psi(h_5)=\frac{n+1}{2}e$.

We distinguish the following two cases to get contradictions to our assumption.

\medskip
\noindent{\bf Case 1.}  $\psi(h_6)=\psi(h_7)=\psi(h_8)=\frac{n+1}{4}e$ if  $n\equiv 3\pmod 4$ and $\psi(h_6)=\psi(h_7)=\psi(h_8)=\frac{3n+1}{4}e$ if  $n\equiv 1\pmod 4$.
\smallskip

Consider the sequence $W=h_2h_4h_6h_7h_8$ (see Figure \ref{f3}).

\begin{figure}[ht]
\begin{center}

\setlength{\unitlength}{0.7 mm}%
\begin{picture}(118.00,34.33)(0,0)
\put(26.32,22.85){\circle*{1.80}}
\put(25.97,2.50){\circle*{1.80}}
\put(46.32,22.85){\circle*{1.80}}
\put(45.98,2.50){\circle*{1.80}}
\put(65.98,22.85){\circle*{1.80}}
\put(85.99,22.85){\circle*{1.80}}
\put(106.34,22.85){\circle*{1.80}}
\put(5.97,12.33){\circle*{1.80}}
\put(2.00,14.28){\fontsize{8.53}{10.24}\selectfont \makebox(7.5, 3.0)[l]{$h_1$\strut}}
\put(22.52,25.00){\fontsize{8.53}{10.24}\selectfont \makebox(7.5, 3.0)[l]{$h_2$\strut}}
\put(22.49,4.96){\fontsize{8.53}{10.24}\selectfont \makebox(7.5, 3.0)[l]{$h_3$\strut}}
\put(42.18,25.00){\fontsize{8.53}{10.24}\selectfont \makebox(7.5, 3.0)[l]{$h_4$\strut}}
\put(41.46,4.79){\fontsize{8.53}{10.24}\selectfont \makebox(7.5, 3.0)[l]{$h_5$\strut}}
\put(61.50,25.00){\fontsize{8.53}{10.24}\selectfont \makebox(7.5, 3.0)[l]{$h_6$\strut}}
\put(82.71,25.00){\fontsize{8.53}{10.24}\selectfont \makebox(7.5, 3.0)[l]{$h_7$\strut}}
\put(103.41,25.00){\fontsize{8.53}{10.24}\selectfont \makebox(7.5, 3.0)[l]{$h_8$\strut}}
\lbezier(16.31,32.33)(16.31,17.33)\lbezier(16.31,17.33)(116.00,17.33)\lbezier(116.00,17.33)(116.00,32.33)\lbezier(116.00,32.33)(16.31,32.33)
\end{picture}%

\end{center}
\caption{}\label{f3}
\end{figure}

Suppose that $n\equiv 3\pmod 4$.
 Then $\psi(W)=0^2(\frac{n+1}{4}e)^3$. Since $\phi(W)\in \mathscr{F}(\phi(G))$ and $\mathsf D(\phi(G))=\mathsf D(C_2^4)=5$, there exists a subsequence $V\mid W$ such that $\sigma(\phi(V))=0$ and $|V|\in \{3, 4, 5\}$. If $|V|=5$, then $\sigma(\psi(V))=3\frac{n+1}{4}e\notin\{e,2e\}$, a contradiction. If $|V|=4$, then $\sigma(\psi(V))=e\in \{2\frac{n+1}{4}e,\,3\frac{n+1}{4}e\}$, a contradiction. Thus $|V|=3$ and  $\sigma(\psi(V))=e\in \{\frac{n+1}{4}e,\,2\frac{n+1}{4}e,3\frac{n+1}{4}e\}$ which implies that $n=3$ and $h_2h_4\t V$. But $\sigma(\phi(Vh_3h_5(h_2h_4)^{-1}))=0$ and $\sigma(\psi(Vh_3h_5(h_2h_4)^{-1})))=2\frac{n+1}{2}e+\frac{n+1}{4}e=2e\neq e$, a contradiction.

Suppose that $n\equiv 1\pmod 4$.
 Then $\psi(W)=0^2(\frac{3n+1}{4}e)^3$. Since $\phi(W)\in \mathscr{F}(\phi(G))$ and $\mathsf D(\phi(G))=\mathsf D(C_2^4)=5$, there exists a subsequence $V\mid W$ such that $\sigma(\phi(V))=0$ and $|V|\in \{3, 4, 5\}$. If $|V|=5$, then $\sigma(\psi(V))=3\frac{3n+1}{4}e\notin\{e,2e\}$ which implies that $n=5$. Since $\sigma(\phi(h_3h_5h_6h_7h_8))=0$, we obtain that $\sigma(\psi(h_3h_5h_6h_7h_8))=3e+3e+4e+4e+4e=3e\notin\{e,2e\}$, a contradiction. If $|V|=4$, then $\sigma(\psi(V))=e\in \{2\frac{3n+1}{4}e,\,3\frac{3n+1}{4}e\}$, a contradiction. Thus $|V|=3$ and  $\sigma(\psi(V))=e\in \{\frac{3n+1}{4}e,\,2\frac{3n+1}{4}e,3\frac{3n+1}{4}e\}$, a contradiction.

\medskip
\noindent{\bf Case 2.}  There exist distinct $i,j\in [6,8]$ such that $\psi(h_i)+\psi(h_j)\neq \frac{n+1}{2}e$, say  $\psi(h_6)+\psi(h_7)\neq \frac{n+1}{2}e$.
\smallskip

  Consider the sequence $W=h_1h_2h_4h_6h_7$ (see Figure \ref{f4}).

  \begin{figure}[ht]
  \begin{center}
 \setlength{\unitlength}{0.7 mm}%
 \begin{picture}(102.66,34.98)(0,0)
 \put(30.73,22.97){\circle*{1.80}}
 \put(30.39,2.62){\circle*{1.80}}
 \put(50.74,22.97){\circle*{1.80}}
 \put(50.39,2.62){\circle*{1.80}}
 \put(70.40,22.97){\circle*{1.80}}
 \put(90.40,22.97){\circle*{1.80}}
 \put(10.38,12.45){\circle*{1.80}}
 \put(7.41,14.40){\fontsize{8.53}{10.24}\selectfont \makebox(7.5, 3.0)[l]{$h_1$\strut}}
 \put(26.76,25.00){\fontsize{8.53}{10.24}\selectfont \makebox(7.5, 3.0)[l]{$h_2$\strut}}
 \put(26.90,5.09){\fontsize{8.53}{10.24}\selectfont \makebox(7.5, 3.0)[l]{$h_3$\strut}}
 \put(46.60,25.00){\fontsize{8.53}{10.24}\selectfont \makebox(7.5, 3.0)[l]{$h_4$\strut}}
 \put(46.87,4.92){\fontsize{8.53}{10.24}\selectfont \makebox(7.5, 3.0)[l]{$h_5$\strut}}
 \put(65.91,25.00){\fontsize{8.53}{10.24}\selectfont \makebox(7.5, 3.0)[l]{$h_6$\strut}}
 \put(86.49,25.00){\fontsize{8.53}{10.24}\selectfont \makebox(7.5, 3.0)[l]{$h_7$\strut}}
 \thinlines\lbezier(20.76,32.78)(100.47,32.78)\lbezier(100.47,32.78)(100.47,18.12)\lbezier(100.47,18.12)(22.50,18.12)\lbezier(22.50,18.12)(10.48,5.64)\lbezier(10.48,5.64)(2.00,13.93)\lbezier(2.00,13.93)(20.76,32.78)
 \end{picture}%

  \end{center}
  \caption{}\label{f4}
  \end{figure}

  Since $\phi(W)\in \mathscr{F}(\phi(G))$ and $\mathsf D(\phi(G))=\mathsf D(C_2^4)=5$, there exists a subsequence $V\mid W$ such that $\sigma(\phi(V))=0$ and $|V|\in \{3, 4, 5\}$. We distinguish three cases depending on $|V|$.

Suppose that $|V|=5$. Then $\sigma(\psi(V))\neq \frac{n+1}{2}e+0+0+\frac{n+1}{2}e=e$ which implies that $\sigma(\psi(V))=2e$. Thus $\sigma(\psi(h_6h_7))=\frac{n+3}{2}e$. It follows that $\sigma(\psi(h_1h_3h_5h_6h_7))=3e\notin\{e,2e\}$, a contradiction to  $\sigma(\phi(h_1h_3h_5h_6h_7))=\sigma(\phi(V))=0$.

Suppose that $|V|=4$. If $h_2h_4\t V$, then $\sigma(\phi(V(h_2h_4)^{-1}h_3h_5))=0$ and $\sigma(\psi(V(h_2h_4)^{-1}h_3h_5))=2e$, a contradiction. Thus $h_2h_4\nmid V$. By symmetry, we only need to consider $V=h_1h_2h_6h_7$. But $\sigma(\psi(V))\neq \frac{n+1}{2}e+0+\frac{n+1}{2}e=e$, a contradiction.

Suppose that $|V|=3$. By symmetry, we only need to consider $V=h_1h_6h_7$, $V=h_2h_6h_7$, $V=h_2h_4h_6$ or $V=h_1h_2h_4$. If $V=h_1h_6h_7$, then $\sigma(\psi(V))\neq e$, a contradiction.
If $V=h_2h_6h_7$, then $\sigma(\phi(h_1h_3h_6h_7))=\sigma(\phi(V))=0$ and hence $\sigma(\psi(h_1h_3h_6h_7))=e$. It follows that $\psi(h_6)+\psi(h_7)=0$ which implies that $\sigma(\psi(V))=0$, a contradiction. If $V=h_2h_4h_6$, then $\sigma(\psi(V))=e$ which implies that $\psi(h_6)=e$. Thus $\sigma(\psi(h_3h_5h_6))=2e$, a contradiction to $\sigma(\phi(h_3h_5h_6))=\sigma(\phi(V))=0$. If $V=h_1h_2h_4$, then $\phi(h_3)=\phi(h_4)$, a contradiction.
\end{proof}

\begin{lemma}\label{L4} Let $G,H,K$ and $\phi,\psi$ be as above.
Let $ K=\left<e\right>$ and $S=h_1\cdot\ldots\cdot h_8$ be a sequence over $G\setminus\{0\}$ with $\phi(h_1)=\phi(h_2)+\phi(h_3)=\phi(h_4)+\phi(h_5)$ and $\psi(h_6)=\psi(h_7)=\frac{n+1}{2}e$.
If  $\phi(S)$  is a squarefree sequence with $0\not\in \supp(\phi(S))$, then the following property $(*)$ does not hold.
  \[
      \left\{\begin{aligned} & \text{For any subsequence $V$ of $S$ with $\sigma(\phi(V))=0$, \text{ we have that }}\\
       &\sigma(\psi(V)) =\left \{ \begin{aligned}&  e ,&\quad \quad \quad& \mbox{ if } |V|=3\mbox{ or }4, \\
      & e\mbox{ or }2e,&\quad &\mbox{ if } |V|=5.
      \end{aligned} \right.
     \end{aligned}\qquad (*)\right.
      \]
\end{lemma}

\begin{proof}
Assume to the contrary that the property $(*)$ holds.

Since  $\phi(h_1)=\phi(h_2)+\phi(h_3)=\phi(h_4)+\phi(h_5)$,  we obtain that $\sigma(\phi(h_1h_2h_3))=\sigma(\phi(h_2h_3h_4h_5))=\sigma(\phi(h_4h_5h_1))=0$ which implies that $\sigma(\psi(h_1h_2h_3))=\sigma(\psi(h_2h_3h_4h_5))=\sigma(\psi(h_4h_5h_1))=e$. Therefore $\psi(h_1)=\psi(h_2)+\psi(h_3)=\psi(h_4)+\psi(h_5)=\frac{n+1}{2}e$.

Let $\psi(h_i)=k_ie$ where $1\le i\le 8$ and $0\le k_i\le n-1$. Without loss of generality, we can assume that $k_2\le k_3,\, k_4\le k_5$. Consider the sequence $W=h_1h_2h_4h_6h_7$ (see Figure \ref{f5}).

\begin{figure}[ht]
\begin{center}

\setlength{\unitlength}{0.7 mm}%
\begin{picture}(102.66,34.98)(0,0)
\put(30.73,22.97){\circle*{1.80}}
\put(30.39,2.62){\circle*{1.80}}
\put(50.74,22.97){\circle*{1.80}}
\put(50.39,2.62){\circle*{1.80}}
\put(70.40,22.97){\circle*{1.80}}
\put(90.40,22.97){\circle*{1.80}}
\put(10.38,12.45){\circle*{1.80}}
\put(7.41,14.40){\fontsize{8.53}{10.24}\selectfont \makebox(7.5, 3.0)[l]{$h_1$\strut}}
\put(26.76,25.00){\fontsize{8.53}{10.24}\selectfont \makebox(7.5, 3.0)[l]{$h_2$\strut}}
\put(26.90,5.09){\fontsize{8.53}{10.24}\selectfont \makebox(7.5, 3.0)[l]{$h_3$\strut}}
\put(46.60,25.00){\fontsize{8.53}{10.24}\selectfont \makebox(7.5, 3.0)[l]{$h_4$\strut}}
\put(46.87,4.92){\fontsize{8.53}{10.24}\selectfont \makebox(7.5, 3.0)[l]{$h_5$\strut}}
\put(65.91,25.00){\fontsize{8.53}{10.24}\selectfont \makebox(7.5, 3.0)[l]{$h_6$\strut}}
\put(86.49,25.00){\fontsize{8.53}{10.24}\selectfont \makebox(7.5, 3.0)[l]{$h_7$\strut}}
\thinlines\lbezier(20.76,32.78)(100.47,32.78)\lbezier(100.47,32.78)(100.47,18.12)\lbezier(100.47,18.12)(22.50,18.12)\lbezier(22.50,18.12)(10.48,5.64)\lbezier(10.48,5.64)(2.00,13.93)\lbezier(2.00,13.93)(20.76,32.78)
\end{picture}%

\end{center}
\caption{}\label{f5}
\end{figure}

 Since $\phi(W)\in \mathscr{F}(\phi(G))$ and $\mathsf D(\phi(G))=\mathsf D(C_2^4)=5$, there exists a subsequence $V\mid W$ such that $\sigma(\phi(V))=0$ and $|V|\in \{3, 4, 5\}$. We distinguish three cases depending on $|V|$.

\medskip
\noindent{\bf Case 1.} $|V|=3$.
\smallskip

By symmetry, we only need to consider $V=h_1h_6h_7$, $V=h_2h_4h_6$ or $V=h_4h_6h_7$.

Suppose that $V=h_1h_6h_7$. Then $\sigma(\psi(V))=\frac{n+3}{2}e\neq e$, a contradiction.

Suppose that $V=h_2h_4h_6$. Then  $\sigma(\phi(h_3h_5h_6))=\sigma(\phi(V))=0$ and hence $\sigma(\psi(h_3h_5h_6))=\sigma(\psi(V))=e$. Thus $\sigma(\psi(h_3h_5))=\sigma(\psi(h_2h_4))=\frac{n+1}{2}e$ which implies that  $\psi(h_3)=\psi(h_4)=\psi(h_5)=\psi(h_6)$ by $k_2\le k_3, k_4\le k_5$. Therefore $\sigma(\psi(h_1h_3h_4h_6))=3\frac{n+1}{2}\neq e$, a contradiction to $\sigma(\phi(h_1h_3h_4h_6))=\sigma(\phi(V))=0$.

Suppose that $V=h_4h_6h_7$. Then  $\sigma(\phi(h_1h_5h_6h_7))=\sigma(\phi(V))=0$ and hence $\sigma(\psi(h_1h_5h_6h_7))=\sigma(\psi(h_4h_6h_7))=e$. Thus $\psi(h_5)=\frac{n-1}{2}e$ and $\psi(h_4)=0$, a contradiction to $\psi(h_4)+\psi(h_5)=\frac{n+1}{2}e$.

\medskip
\noindent{\bf Case 2.} $|V|=4$.
\smallskip

By symmetry, we only need to consider $V=h_2h_4h_6h_7$, $V=h_1h_4h_6h_7$ or $V=h_1h_2h_4h_6$.

Suppose that $V=h_2h_4h_6h_7$. Then $\sigma(\phi(h_3h_5h_6h_7))=\sigma(\phi(V))=0$ and hence $\sigma(\psi(h_3h_5h_6h_7))=\sigma(\psi(V))=e$. Thus $\sigma(\psi(h_3h_5))=\sigma(\psi(h_2h_4))=0$, a contradiction to $\sigma(\psi(h_2h_3h_4h_5))=e$.

Suppose that $V=h_1h_4h_6h_7$. Then $\sigma(\phi(h_5h_6h_7))=\sigma(\phi(V))=0$ and hence $\sigma(\psi(h_5h_6h_7))=e$. Thus $\psi(h_5)=0$, a contradiction to $k_5\ge k_4$ and $\psi(h_4)+\psi(h_5)=\frac{n+1}{2}e$.

Suppose that $V=h_1h_2h_4h_6$. Then  $\sigma(\phi(h_1h_3h_5h_6))=\sigma(\phi(V))=0$ and hence $\sigma(\psi(h_1h_3h_5h_6))=\sigma(\psi(V))=e$. Thus $\sigma(\psi(h_3h_5))=\sigma(\psi(h_2h_4))=0$, a contradiction to $\sigma(\psi(h_2h_3h_4h_5))=e$.

\medskip
\noindent{\bf Case 3.} $|V|=5$.
\smallskip

Then $\sigma(\phi(h_3h_4h_6h_7))=\sigma(\phi(h_2h_5h_6h_7))=\sigma(\phi(V))=0$ which implies that $\sigma(\psi(h_3h_4h_6h_7))=\sigma(\psi(h_2h_5h_6h_7))=e$. Thus $\sigma(\psi(h_3h_4))=\sigma(\psi(h_2h_5))=0$ which implies that $\sigma(\psi(h_2h_3h_4h_5))=0$, a contradiction to $\sigma(\phi(h_2h_3h_4h_5))=0$.
\end{proof}

\begin{prop}\label{LL1} Let $G,H,K$ and $\phi,\psi$ be as above.
Let $ K=\left<e\right>$ and $S$ be a sequence over $G\setminus\{0\}$.
If  $\phi(S)$  is a squarefree sequence of length $|\phi(S)|=8$ with $0\not\in \supp(\phi(S))$ and satisfies the following property $(*)$:
 \[
     \left\{\begin{aligned} & \text{For any subsequence $V$ of $S$ with $\sigma(\phi(V))=0$, \text{ we have that }}\\
      &\sigma(\psi(V)) =\left \{ \begin{aligned}&  e ,&\quad \quad \quad& \mbox{ if } |V|=3\mbox{ or }4, \\
     & e\mbox{ or }2e,&\quad &\mbox{ if } |V|=5.
     \end{aligned} \right.
    \end{aligned}\qquad (*)\right.
     \]
  then $\frac{n+1}{2}e\not\in\supp(\psi(S))$.
\end{prop}

\begin{proof}

For any $v\in \phi(G)\setminus \{0\}=H\setminus\{0\}$, we define
$$\mathsf N_v(\phi(S))=|\{T\t \phi(S) \ :\  |T|=2 \text{ and } \sigma(T)=v  \}|+\delta_v,$$
where
\[
\delta_v =\left \{ \begin{array}{ll}  1 ,\quad \mbox{ if } v\in \mathsf{supp}(\phi(S)); \\  0,\quad \mbox{ if } v\notin \mathsf{supp}(\phi(S)).
\end{array} \right.
\]

 We distinguish the following four cases.

\medskip\noindent{\bf Case 1.} There exists $v\in H\setminus\{0\}$
 such that $\mathsf N_v(\phi(S))=4$ and $\delta_v=0$.
 \smallskip

Without loss of generality, we can assume that $v=\phi(h_1+h_2)=\phi(h_3+h_4)=\phi(h_5+h_6)=\phi(h_7+h_8)$ where $S=h_1\cdot\ldots\cdot h_8$.  By Lemma \ref{L1}, we obtain that  $\frac{n+1}{2}e\not\in\supp(\psi(S))$.

\medskip\noindent{\bf Case 2.} There exists $v\in H\setminus \{0\}$
 such that $\mathsf N_v(\phi(S))=4$ and $\delta_v=1$.
 \smallskip

Without loss of generality, we can assume that $v=\phi(h_1)=\phi(h_2+h_3)=\phi(h_4+h_5)=\phi(h_6+h_7)$ where $S=h_1\cdot\ldots\cdot h_8$.  By Lemma \ref{L2}, we obtain that  $\frac{n+1}{2}e\not\in\supp(\psi(S))$.

\medskip
Now, we can assume that, for each $v\in H\setminus \{0\}$,
 $\mathsf N_v(\phi(S))\le 3$.
Since $\sum_{v\in H\setminus \{0\}}\mathsf N_v(\phi(S))=\frac{8\times 7}{2}+8=36$ and $|H\setminus \{0\}|=15$, by simple calculation, we obtain that $|\{v\in H\setminus \{0\}\mid \mathsf N_v(\phi(S))= 3\}|\ge 6$. We continue with further case distinctions.

\medskip\noindent{\bf Case 3.} There exist three distinct $v_1, v_2, v_3\in H\setminus \{0\}$
 such that $\mathsf N_{v_1}(\phi(S))=\mathsf N_{ v_2}(\phi(S))=\mathsf N_{ v_3}(\phi(S))=3$ and $\delta_{ v_1}=\delta_{ v_2}=\delta_{ v_3}=1$.
 \smallskip

Let $S=h_1\cdot\ldots\cdot h_8$. For each $i\in [1,3]$, we  denote by $A_i=\{ v_i\} \cup \supp(\phi(T_{i_1}T_{i_2}))$, where $|T_{i_1}|=|T_{i_2}|=2$, $T_{i_1}T_{i_2}\t S$ , and $ v_i=\sigma(\phi(T_{i_1}))=\sigma(\phi(T_{i_2}))$. Thus $|A_i|=5$ for each $i\in[1,3]$. By symmetry, we can distinguish the following two cases.

\medskip
\noindent{\bf Subcase 3.1. }There exists  $i\in [1,3]$, say $i=1$,  such that $ v_2\notin A_1$ and $ v_3\notin A_1$.
\smallskip

Then we can assume that $ v_1=\phi(h_1)=\phi(h_2+h_3)=\phi(h_4+h_5)$, $ v_2=\phi(h_6)$, and $v_3=\phi(h_7)$. It follows that $\psi(h_1)=\psi(h_6)=\psi(h_7)=\frac{n+1}{2}e$,  a contradiction to Lemma \ref{L4}.

\medskip
\noindent{\bf  Subcase 3.2.}For each $i\in [1,3]$, there exists $j\in [1,3]$ such that $j\neq i$ and $ v_j\in A_i$.
\smallskip

For $i=1$, we can assume that $ v_2\in A_1$. Then $ v_1\in A_2$.
For $i=3$, we obtain that $ v_1\in A_3$ or $ v_2\in A_3$ which implies that $ v_3\in A_1$ or $ v_3\in A_2$. By symmetry, we can assume that $ v_3\in A_1$ and hence $ v_2, v_3\in A_1$.

Without loss of generality, we can assume that  $ v_1=\phi(h_1)=\phi(h_2+h_3)=\phi(h_4+h_5)$, $ v_2=\phi(h_3)$, and $ v_3=\phi(h_5)$. It follows that $\psi(h_1)=\psi(h_3)=\psi(h_5)=\frac{n+1}{2}e$, a contradiction to Lemma \ref{L3}.

\medskip\noindent{\bf Case 4.} There exist three distinct $ v_1, v_2, v_3\in H\setminus \{0\}$
 such that $\mathsf N_{ v_1}(\phi(S))=\mathsf N_{ v_2}(\phi(S))=\mathsf N_{ v_3}(\phi(S))=3$ and $\delta_{ v_1}=\delta_{ v_2}=\delta_{ v_3}=0$.
 \smallskip

Let $S=h_1\cdot\ldots\cdot h_8$. For each $i\in [1,3]$, we  denote by $A_i=\supp(\phi(R_{i_1}R_{i_2}R_{i_3}))$, where $|R_{i_1}|=|R_{i_2}|=|R_{i_3}|=2$, $R_{i_1}R_{i_2}R_{i_3}\t S$ , and $ v_i=\sigma(\phi(R_{i_1}))=\sigma(\phi(R_{i_2}))=\sigma(\phi(R_{i_3}))$. Thus $|A_i|=6$ for each $i\in[1,3]$ and hence $|A_i\cap A_j|\ge 6+6-8= 4$ for distinct $i,j$ where $1\le i,j\le 3$.
We proceed by the following two claims

\begin{enumerate}
 \item[]\noindent{\bf Claim A. } For each $i\in [1,3]$ and each $k\in [1,3]$, $\sigma(\psi(R_{i_k}))=\frac{n+1}{2}e$.
\smallskip

\noindent
{\it Proof of \,{\bf Claim A}. }
For each $i\in [1,3]$, $\sigma(\phi(R_{i_1}R_{i_2}))=\sigma(\phi(R_{i_1}R_{i_3}))=\sigma(\phi(R_{i_2}R_{i_3}))=0$ implies that $\sigma(\psi(R_{i_1}R_{i_2}))=\sigma(\psi(R_{i_1}R_{i_3}))=\sigma(\psi(R_{i_2}R_{i_3}))=e$. Thus $\sigma(\psi(R_{i_k}))=\frac{n+1}{2}e$ for all $k\in [1,3]$.

\qedhere{(Proof of Claim A)}

\item[]
\noindent{\bf Claim B. } For each $j\in [2,3]$, there exist $1\le s<t\le 3$ such that $\supp(\phi(R_{1_s}R_{1_t}))\subseteq A_j$.
Furthermore, there exist distinct $1\le x,y\le 3$ such that $R_{1_s}R_{1_t}=R_{j_x}R_{j_y}$.

\smallskip

\noindent
{\it Proof of \,{\bf  Claim B}. }
Without loss of generality, we can assume that $j=2$. Let $R_{1_1}=g_1g_2, R_{1_2}=g_3g_4$, and $ R_{1_3}=g_5g_6$.

Since $|A_1\cap A_2|\ge  4$, by symmetry, we only need to consider two cases: $\supp(\phi(g_1g_2g_3g_4))\subseteq A_1\cap A_2$ and $\supp(\phi(g_1g_2g_3g_5))= A_1\cap A_2$.

Suppose that $\supp(\phi(g_1g_2g_3g_5)) =A_1\cap A_2$. Then there exists $x\in [1,3]$ such that $R_{2_x}\t g_1g_2g_3g_5$. By symmetry, there are only three cases: $R_{2_x}=g_1g_2$, $R_{2_x}=g_1g_3$, and $R_{2_x}=g_3g_5$. If $R_{2_x}=g_1g_2$, then $ v_1= v_2$, a contradiction.
If $R_{2_x}=g_1g_3$, then  $ v_2=\sigma(\phi(R_{2_x}))=\sigma(\phi(g_2g_4))$ which implies that $\phi(g_4)\in A_1\cap A_2$, a contradiction.
If $R_{2_x}=g_3g_5$, then  $ v_2=\sigma(\phi(R_{2_x}))=\sigma(\phi(g_4g_6))$ which implies that $\supp(\phi(g_4g_6))\subseteq A_1\cap A_2$, a contradiction.

Suppose that $\supp(\phi(g_1g_2g_3g_4))\subseteq A_1\cap A_2$. Then $\supp(\phi(R_{1_1}R_{1_2}))\subseteq A_2$. Furthermore, there must exist $x\in[1,3]$ such that $R_{2_x}\t R_{1_1}R_{1_2}$. Thus $\sigma(\phi(R_{2_x}))=\sigma(\phi(R_{1_1}R_{1_2}R_{2_x}^{-1}))$ which implies that $R_{1_1}R_{1_2}R_{2_x}^{-1}=R_{2_y}$ for some $y\in [1,3]\setminus \{x\}$.
\qed{(Proof of Claim B)}
\end{enumerate}
\medskip

Without loss of generality, we can assume that $S=h_1\cdot\ldots\cdot h_8$, $ v_1=\phi(h_1)+\phi(h_2)=\phi(h_3)+\phi(h_4)=\phi(h_5)+\phi(h_6)$.

If $|A_1\cap A_2|=|A_1\cap A_3|=4$, then $ v_2=\sigma(\phi(h_7h_8))= v_3$, a contradiction. Thus by symmetry and Claim B, we can assume that $|A_1\cap A_2|=5$ and $ v_2=\phi(h_1)+\phi(h_3)=\phi(h_2)+\phi(h_4)=\phi(h_5)+\phi(h_7)$ which  implies that   $\phi(h_1)+\phi(h_4)=\phi(h_2)+\phi(h_3)=\phi(h_6)+\phi(h_7)$
(See Figure \ref{f6}).
\smallskip

\begin{figure}[ht]
\begin{center}
\setlength{\unitlength}{1.0 mm}%
\begin{picture}(87.73,40.58)(0,0)
\put(7.64,30.22){\circle*{1.80}}
\put(7.36,10.35){\circle*{1.80}}
\put(25.69,30.36){\circle*{1.80}}
\put(25.55,10.35){\circle*{1.80}}
\put(57.31,30.64){\circle*{1.80}}
\put(81.49,30.42){\circle*{1.80}}
\put(67.80,10.07){\circle*{1.80}}
\put(5.00,35.47){\fontsize{8.53}{10.24}\selectfont \makebox(7.5, 3.0)[l]{$h_1$\strut}}
\put(5,1.76){\fontsize{8.53}{10.24}\selectfont \makebox(7.5, 3.0)[l]{$h_2$\strut}}
\put(24.20,35.47){\fontsize{8.53}{10.24}\selectfont \makebox(7.5, 3.0)[l]{$h_3$\strut}}
\put(24.20,1.76){\fontsize{8.53}{10.24}\selectfont \makebox(7.5, 3.0)[l]{$h_4$\strut}}
\put(45.21,1.76){\fontsize{8.53}{10.24}\selectfont \makebox(7.5, 3.0)[l]{$h_6$\strut}}
\put(68.30,1.76){\fontsize{8.53}{10.24}\selectfont \makebox(7.5, 3.0)[l]{$h_7$\strut}}
\put(78.23,34.20){\fontsize{8.53}{10.24}\selectfont \makebox(7.5, 3.0)[l]{$h_8$\strut}}
\put(46.95,10.21){\circle*{1.80}}
\put(52.64,34.77){\fontsize{8.53}{10.24}\selectfont \makebox(7.5, 3.0)[l]{$h_5$\strut}}
\lbezier(4.00,33.65)(4.00,6.36)\lbezier(4.00,6.36)(8.76,6.36)\lbezier(8.76,6.36)(8.76,33.65)\lbezier(8.76,33.65)(4.00,33.65)
\lbezier(24.22,33.65)(24.22,6.43)\lbezier(24.22,6.43)(28.98,6.43)\lbezier(28.98,6.43)(28.98,33.65)\lbezier(28.98,33.65)(24.22,33.65)
\lbezier(55.63,34.84)(43.62,9.52)\lbezier(43.62,9.52)(48.27,7.31)\lbezier(48.27,7.31)(60.28,32.63)\lbezier(60.28,32.63)(55.63,34.84)

\dashline[33]{0.50}(4.07,33.51)(4.07,29.10)\dashline[33]{0.50}(4.07,29.10)(28.98,29.10)\dashline[33]{0.50}(28.98,29.10)(28.98,33.51)\dashline[33]{0.50}(28.98,33.51)(4.07,33.51)
\dashline[33]{0.50}(3.93,11.82)(3.93,6.36)\dashline[33]{0.50}(3.93,6.36)(28.98,6.36)\dashline[33]{0.50}(28.98,6.36)(28.98,11.82)\dashline[33]{0.50}(28.98,11.82)(3.93,11.82)
\dashline[33]{0.50}(67.59,6.16)(54.39,31.49)\dashline[33]{0.50}(54.39,31.49)(58.25,33.50)\dashline[33]{0.50}(58.25,33.50)(71.44,8.16)\dashline[33]{0.50}(71.44,8.16)(67.59,6.16)

\thinlines
\dottedline{0.50}(43.81,11.82)(43.80,7.40)\dottedline{0.50}(43.80,7.40)(69.55,7.34)\dottedline{0.50}(69.55,7.34)(69.56,11.77)\dottedline{0.50}(69.56,11.77)(43.81,11.82)
\dottedline{0.50}(4.14,29.17)(24.22,6.57)\dottedline{0.50}(24.22,6.57)(29.12,11.68)\dottedline{0.50}(29.12,11.68)(8.76,33.72)\dottedline{0.50}(8.76,33.72)(4.14,29.17)
\dottedline{0.50}(24.15,33.72)(28.98,29.38)\dottedline{0.50}(28.98,29.38)(8.83,6.50)\dottedline{0.50}(8.83,6.50)(3.93,11.75)\dottedline{0.50}(3.93,11.75)(24.15,33.72)
\end{picture}%

\end{center}
\caption{}\label{f6}
\end{figure}

Then we have \begin{align*}
\psi(h_1\cdot\ldots\cdot h_7)&=(\frac{n+1}{4}e)^7 \quad \text{ if } n\equiv 3 \pmod 4\,,\\
\psi(h_1\cdot\ldots\cdot h_7)&=(\frac{3n+1}{4}e)^7 \quad \text{ if } n\equiv 1 \pmod 4\,.\\
\end{align*}
Assume to the contrary that $\frac{n+1}{2}e\in \mathsf{supp}(\psi(S))$. Then
 $\psi(h_8)=\frac{n+1}{2}e$.

Consider the sequence $W=h_1h_2h_5h_7h_8$. Since $\phi(W)\in \mathscr{F}(\phi(G))$ and $\mathsf D(\phi(G))=\mathsf D(C_2^4)=5$, there exists a subsequence $V\mid W$ such that $\sigma(\phi(V))=0$ and $|V|\in \{3, 4, 5\}$.  If $|V|=4$, by $\sigma(\psi(V))=e$ we obtain that $V=h_1h_2h_5h_7$ which implies that $\phi(h_6)=\phi(h_7)$, a contradiction. If $|V|=5$,  then  $\sigma(\psi(V))=\frac{n+3}{2}e\notin \{e, 2e\}$, a contradiction.

Therefore $|V|=3$. By $\sigma(\psi(V))=1$, we obtain that $h_8\t V$. Since $h_1h_2\nmid V $ and $h_5h_7\nmid V$,  by symmetry, we only need to consider $V=h_1h_5h_8$. Then $\sigma(\phi(h_2h_3h_4h_5h_8))=0$ and $\sigma(\psi(h_2h_3h_4h_5h_8))\notin \{e, 2e\}$, a contradiction.

\end{proof}

\begin{cor}\label{xx1} Let $G,H,K$ and $\phi,\psi$ be as above.
Let $ K=\left<e\right>$ and $S$ be a sequence over $G\setminus\{0\}$.
If  $\phi(S)$  is a squarefree sequence  with $0\not\in \supp(\phi(S))$ and satisfies the following property $(*)$:
 \[
     \left\{\begin{aligned} & \text{For any subsequence $V$ of $S$ with $\sigma(\phi(V))=0$, \text{ we have that }}\\
      &\sigma(\psi(V)) =\left \{ \begin{aligned}&  e ,&\quad \quad \quad& \mbox{ if } |V|=3\mbox{ or }4, \\
     & e\mbox{ or }2e,&\quad &\mbox{ if } |V|=5.
     \end{aligned} \right.
    \end{aligned}\qquad (*)\right.
     \]
  then $|S|\le 8$.
\end{cor}

\begin{proof}
Assume to the contrary that $|S|\ge 9$. Without loss of generality, we can assume that $|S|=9$.

If $\supp(\psi(S))=\{0\}$, then $S\in \mathscr F(\ker(\psi))$. By $\mathsf D(\ker(\psi))=\mathsf D(C_2^4)=5$, $S$ has a  subsequence $V$ of length $|V|\in [3,5]$ such that $\sigma(\phi(V))=0$ and $\sigma(\psi(V))=0$, a contradiction.

 Thus we can choose $w\t S$ such that $\psi(w)\neq0$. By Lemma \ref{SUM}, there exist distinct  $g_1,g_2\in \supp(Sw^{-1})$ such that $\phi(w)=\phi(g_1)+\phi(g_2)$.
If there exist another $g_3,g_4\in \supp(S(wg_1g_2)^{-1})$ such that $\phi(w)=\phi(g_3)+\phi(g_4)$, then $\sigma(\phi(wg_1g_2))=\sigma(\phi(wg_3g_4))=\sigma(\phi(g_1g_2g_3g_4))=0$ which implies that $\sigma(\psi(wg_1g_2))=\sigma(\psi(wg_3g_4))=\sigma(\psi(g_1g_2g_3g_4))=e$. Therefore $\psi(w)=\frac{n+1}{2}e$, a contradiction to Proposition \ref{LL1}.

Otherwise, choose $g_3\in \supp(S(wg_1g_2)^{-1})$. Then $\phi(S(w+g_3)w^{-1})$ is a squarefree sequence of length $9$. By Lemma \ref{SUM},  there exist distinct  $g_i,g_j\in \supp(Sw^{-1})$ such that $\phi(w+g_3)=\phi(g_i)+\phi(g_j)$. Clearly, $g_ig_j\t S(wg_1g_2g_3)^{-1}$ or $|\{g_1,g_2\}\cap\{g_i,g_j\}|=1$.

Suppose that $g_ig_j\t S(wg_1g_2g_3)^{-1}$. Then $\sigma(\phi(wg_1g_2))=\sigma(\phi(wg_3g_ig_j))=\sigma(\phi(g_1g_2g_3g_ig_j))=0$ which implies that $\sigma(\psi(wg_1g_2))=\sigma(\psi(wg_3g_ig_j))=e$ and $\sigma(\psi(g_1g_2g_3g_ig_j))\in \{e,2e\}$. Therefore $\psi(w)\in\{\frac{n+1}{2}e, 0\}$. It follows by the choice of $w$  that $\psi(w)=\frac{n+1}{2}e$, a contradiction to Proposition \ref{LL1}.

Supppose that  $|\{g_1,g_2\}\cap\{g_i,g_j\}|=1$. By symmetry, we can assume that $g_i=g_1$ and $g_j\in \supp(S(wg_1g_2g_3)^{-1})$. Then $\sigma(\phi(wg_1g_2))=\sigma(\phi(wg_1g_3g_j))=\sigma(\phi(g_1g_2g_1g_3g_j))=\sigma(\phi(g_2g_3g_j))=0$ which implies that $\sigma(\psi(wg_1g_2))=\sigma(\psi(wg_1g_3g_j))=\sigma(\psi(g_2g_3g_j))=e$. Therefore $\psi(g_2)=\frac{n+1}{2}e$, a contradiction to Proposition \ref{LL1}.
\end{proof}

\section{The proof of Theorem \ref{Th1}.2}
\begin{prop}\label{RE1} $\eta(C_2^3\oplus C_{2n})=2n+6$, where $n\geq  3$ is an odd integer.
\end{prop}

\begin{proof} Let $G=H\oplus K$ be a finite abelian group, where $H\cong C_2^4$  and $K\cong C_n$ with $n\geq 3$ an odd integer. Denote $\phi$ to be the projection from  $G$ to  $H$ and $\psi$ to be the projection from $ G$ to  $K$.

  In order to prove that $\eta(G)=2n+6$, by Lemma \ref{ETA} we only need to prove that $\eta(G)\le 2n+6$.
 Assume to the contrary that there exists a sequence  $S$ of length $2n+6$ over $G$  containing no short zero-sum subsequence.

Since $|\phi(S)|=|S|=2n+6=2(n-5)+16$, $\eta(C_2^4)=16$, and $\mathsf D(C_n)=n$,
we obtain that $S$ allows a product decomposition as
$$S=S_1\cdot \ldots \cdot S_{r}\cdot S_0,$$
where $S_1, \ldots, S_{r}, S_0$ are sequences over $G$ and, for every $i\in [1, r]$, $\phi(S_i)$ has sum zero and length $|S_i|\le 2$. Therefore $\phi(S_0)$ is squarefree over $H\setminus\{0\}$ and $n-4\leq r\leq n-1$.
We distinguish the following four cases depending on $r$ to get contradictions.

\medskip

\noindent\textbf{Case 1.} $r=n-1$. Then $|S_0|\ge 8$.
\smallskip

We proceed by the following  assertion first

\begin{itemize}
\item[]
\noindent
{\bf Assertion $A$. }There exists an element $e\in \ker(\phi)=K$ such that $\sigma(S_1)\cdot\ldots\cdot \sigma(S_{n-1})=e^{n-1}$.
  Furthermore, for any element $h\t SS_0^{-1}$, the sequence $S_0h$ has the following property:
 \[
     \left\{\begin{aligned} & \text{For any subsequence $V$ of $S_0h$ with $\sigma(\phi(V))=0$, \text{ we have that }}\\
      &\sigma(\psi(V)) =\left \{ \begin{aligned}&  e ,&\quad \quad \quad& \mbox{ if } |V|=3\mbox{ or }4, \\
     & e\mbox{ or }2e,&\quad &\mbox{ if } |V|=5.
     \end{aligned} \right.
    \end{aligned}\right.
     \]

\medskip
\noindent
{\it Proof of \,{\bf  Assertion A}. } By our assumption and  Lemma \ref{cyclic}.1, $\sigma(S_1)\cdot \ldots\cdot \sigma(S_{n-1})$ is zero-sum free over $K$. Then there exists an element $e\in K\setminus \{0\} $ such that  $\sigma(S_1)=\cdots =\sigma(S_{n-1})=e$.

Without loss of generality, we can assume that $h\t S_{n-1}$.
If $\sigma(\psi(V))=0$, then $V$ is a  short zero-sum subsequence of $S$, a contradiction. Thus $\sigma(\psi(V))\neq0$.

If $\sigma(\psi(V))=ke$ with $k\in [2,n-1]$ and $|V|\in \{3,4\}$, then
$S_1\cdot \ldots \cdot S_{n-k}\cdot V$ is a short zero-sum subsequence of $S$ , a contradiction.

If $\sigma(\psi(V))=ke$ with $k\in [3,n-1]$ and $|V|=5$, then
$S_1\cdot \ldots \cdot S_{n-k}\cdot V$ is a short zero-sum subsequence of $S$ , a contradiction.
\qed{(Proof of Assertion A)}
\end{itemize}
\medskip

If $|S_0|> 8$, we obtain a contradiction to Corollary \ref{xx1}. Thus we can assume that $|S_0|=8$ and hence $|S_i|=2$ for each $i\in [1,n-1]$.

If $\supp(\phi(S))\nsubseteq \supp(\phi(S_0))$, there exists $h\t SS_0^{-1}$,  such that $\phi(S_0h)$ is squarefree. By Corollary \ref{xx1}, $|S_0h|\le 8$, a contradiction.
Thus $\supp(\phi(S))\subseteq \supp(\phi(S_0))$. Without loss of generality, we can assume that $S_{n-1}=h_1h_2$, $h_3\t S_0$ and $\phi(h_1)=\phi(h_2)=\phi(h_3)$. If there exist distinct $1\le i,j\le 3$ such that $\psi(h_i+h_j)\neq e$, then $S_1\cdot\ldots\cdot S_{n-2}\cdot h_ih_j$ has a short zero-sum subsequence, a contradiction. Therefore
$\psi(h_1+h_2)=\psi(h_1+h_3)=\psi(h_2+h_3)=e$ which implies that $\psi(h_1)=\psi(h_2)=\psi(h_3)=\frac{n+1}{2}e$ and hence $\frac{n+1}{2}e\in \supp(\phi(S_0))$, a contradiction to Proposition \ref{LL1}.

\medskip

\noindent\textbf{Case 2.} $r=n-2$. Then $|S_0|\ge 10$.
\smallskip

 By Lemma \ref{IMP},  $S_0$ has a subsequence  $T$ of length $|T|\in \{3,4\}$  such that $\sigma(\phi(T))=0$ and $\sigma(\psi(T))\neq \sigma(\psi(S_1))$. By our assumption,  the sequences $\sigma(S_1)\cdot \ldots\cdot \sigma(S_{n-2})\sigma(T)$ is zero-sum free over $\ker(\phi)=K$.  By Lemma \ref{cyclic}.1, we obtain that
$$\sigma(S_1)=\cdots=\sigma(S_{n-2})=\sigma(T),$$
 a contradiction.

\medskip

\noindent\textbf{Case 3.} $r=n-3$. Then $|S_0|\ge 12$.
\smallskip

If $n=3$, then by Lemma \ref{SHO}, $S$ contains a short zero-sum subsequence, a contradiction.  We can assume that $n\ge 5$.

\smallskip
By Lemma \ref{SUM}, there exist disjoint $T_1, T_2\mid S_0$ such that $\sigma(\phi(T_1))=\sigma(\phi(T_2))=0$ and $|T_1|=|T_2|=3$. By our assumption, the sequence $\sigma(S_1)\cdot \ldots \cdot \sigma(S_{n-3})\cdot \sigma(T_1)\cdot \sigma(T_2)$ contains no zero-sum subsequence over $\ker(\phi)=K\cong C_n$, therefore by Lemma \ref{cyclic}.1,
$$\sigma(S_1)=\cdots=\sigma(S_{n-3})=\sigma(T_1)=\sigma(T_2)=e,$$ for some $e\in \ker(\phi)=K$ of order $n$.

\begin{itemize}
\item[]
\noindent\textbf{Assertion $B$. } Let $V$ be a subsequence of $S_0$ with $\sigma(\phi(V))=0$. Then
\[
\sigma(\psi(V)) =\left \{ \begin{aligned}
  & e ,&& \mbox{ if } |V|=3, \\
  & e\mbox{ or }2e,&\quad &\mbox{ if } |V|=4\mbox{ or }5.
\end{aligned} \right.
\]

\noindent
{\it Proof of \,{\bf  Assertion B}. } If $|V|=3$, then $|S_0V^{-1}|=12-3=9$.
By Lemma \ref{SUM}, there exists $V_1\mid S_0V^{-1}$ such that $\sigma(\phi(V_1))=0$ and $|V_1|=3$.
By our assumption, the sequence $\sigma(S_1)\cdot \ldots \cdot \sigma(S_{n-3})\cdot \sigma(V)\cdot \sigma(V_1)$ contains no zero-sum subsequence over $K$. Therefore by Lemma \ref{cyclic}.1,
$$\sigma(S_1)=\cdots=\sigma(S_{n-3})=\sigma(V)=\sigma(V_1)=e.$$

If $|V|=4$ or $5$, by our assumption,  $\sigma(S_1)\cdot \ldots\cdot \sigma(S_{n-3})\sigma(V)$ is zero-sum free over $K$. Since $\sigma(S_1)=\cdots=\sigma(S_{n-3})=e$, we obtain that   $\Sigma(\sigma(S_1)\cdot \ldots\cdot\sigma(S_{n-3}))=\{e,\ldots,(n-3)e\}$.  It follows that  $\sigma(\psi(V))\in \{e, 2e\}$.
\qed{(Proof of Assertion B)}
\end{itemize}

Suppose that $\mathsf{supp}(\psi(S_0))\setminus \{0, \frac{n+1}{2}e\}\neq \emptyset$.
Choose $u\mid S_0$ such that $\psi(u)\notin \{0, \frac{n+1}{2}e\}$. By Lemma \ref{SUM}, there exists a set $\{u_1, u_2, u_3, u_4\}\subseteq \mathsf{supp}(S_0u^{-1})$ such that $\sigma(\phi(uu_1u_2))=\sigma(\phi(uu_3u_4))=\sigma(\phi(u_1u_2u_3u_4))=0$. Then by Assertion $B$, we deduce that $\sigma(\psi(uu_1u_2))=\sigma(\psi(uu_3u_4))=e$ and  $\sigma(\psi(u_1u_2u_3u_4))\in \{e, 2e\}$. Therefore  $\psi(u_1+u_2)=\psi(u_3+u_4)\in \{e, \frac{n+1}{2}e\}$ and hence  $\psi(u)\in \{0, \frac{n+1}{2}e\}$, a contradiction.

\smallskip

Suppose that $\mathsf{supp}(\psi(S_0))\subseteq \{0, \frac{n+1}{2}e\}$.
If there exists $v\mid S_0$ such that $\psi(v)=\frac{n+1}{2}e$, by Lemma \ref{SUM}, there exists a set  $\{ v_1, \ldots,  v_8\}\subseteq \mathsf{supp}(S_0v^{-1})$ such that
$$\sigma(\phi(v v_1 v_2))=\sigma(\phi(v v_3 v_4))=\sigma(\phi(v v_5 v_6))=\sigma(\phi(v v_7 v_8))=0.$$
Thus
$$\sigma(\psi(vv_1v_2))=\sigma(\psi(v v_3 v_4))=\sigma(\psi(v v_5 v_6))=\sigma(\psi(vv_7 v_8))=e$$
and
$$\psi( v_1+ v_2)=\psi( v_3+ v_4)=\psi( v_5+v_6)=\psi( v_7+ v_8)=\frac{n+1}{2}e.$$
Since $\mathsf{supp}(\psi(S_0))\subseteq \{0, \frac{n+1}{2}e\}$, we have $\psi( v_1\cdot \ldots \cdot  v_8)=0^4(\frac{n+1}{2}e)^4$ which implies that $0^4\t \psi(S_0)$. Then we can always assume that $0^4\t \psi(S_0)$.

Choose $R\mid S_0$ such that $0^4\t\psi(R)$ and $|R|=5$.
By $\mathsf D(C_2^4)=5$, there exists $R_1\t R$ such that $\sigma(\phi(R_1))=0$. By our assumption, $\sigma(\psi(R_1))\neq 0$. It follows that $\sigma(\psi(R_1))=\frac{n+1}{2}e\notin \{e, 2e\}$ by $n\ge 5$, a contradiction.

\medskip
\noindent\textbf{Case 4.} $r=n-4$. Then $|S_0|\ge14$ and $n\ge 5$.
\smallskip

By Lemma \ref{SUM}, there exists a subsequence $T_1$ of $S_0$ such that $\sigma(\phi(T_1))=0$ and $|T_1|=3$. Since $|S_0T_1^{-1}|=11$, there exists a subsequence $T_2$ of $S_0T_1^{-1}$ such that   $\sigma(\phi(T_2))=0$, $|T_2|\in \{3,4\}$, and $\sigma(\psi(T_2))\neq \sigma(\psi(T_1))$ by Lemma \ref{IMP}. By our assumption, the sequence $\sigma(S_1)\cdot \ldots\cdot \sigma(S_{n-4})\sigma(T_1)\sigma(T_2)$ contains no zero-sum subsequence. Therefore by Lemma \ref{cyclic}.2, there exists an element $e\in K$ such  that
$$\sigma(S_1)\cdot\ldots\cdot\sigma(S_{n-4})\cdot\sigma(T_1)\cdot\sigma(T_2)=e^{n-3}(2e),$$
which implies that  $\sigma(S_1)=\ldots=\sigma(S_{n-4})=e$.

Again by Lemma \ref{IMP}, there exists a subsequence $T_3$ of $S_0$ such that $\sigma(\phi(T_3))=0$, $|T_3|\in \{3,4\}$, and $\sigma(\psi(T_3))\neq e$. Therefore $\sigma(\psi(T_3))=2e$ or $3e$.

Suppose that $\sigma(\psi(T_3))=2e$.
Since $|S_0T_3^{-1}|\ge 10$, there exists a subsequence $T_4$ of $S_0T_3^{-1}$ such that  $\sigma(\phi(T_4))=0$, $|T_4|\in \{3,4\}$, and $\sigma(\psi(T_4))=te$ with $t\in[2,n]$. If $t\ge 4$, then  $S_1\cdot \ldots\cdot S_{n-t}\cdot T_4$ is a  short zero-sum subsequence of $S$, a contradiction. Otherwise $2\le t\le 3$.  Then  $S_1\cdot \ldots\cdot S_{n-t-2}\cdot T_3\cdot T_4$ is a  short zero-sum subsequence of $S$, a contradiction.

Suppose that $\sigma(\psi(T_3))=3e$.
Since $|S_0T_3^{-1}|\ge 10$, there exists a subsequence $T_4$ of $S_0T_3^{-1}$ such that  $\sigma(\phi(T_4))=0$, $|T_4|\in \{3,4\}$, and $\sigma(\psi(T_4))=te$ with $t\in[1,n]\setminus\{3\}$. If $t\ge 4$, then  $S_1\cdot \ldots\cdot S_{n-t}\cdot T_4$ is a  short zero-sum subsequence of $S$, a contradiction. Otherwise $1\le t\le 2$. Then  $S_1\cdot \ldots\cdot S_{n-3-t}\cdot T_3\cdot T_4$ is a  short zero-sum subsequence of $S$, a contradiction.
\end{proof}

\begin{lemma}\label{EIMP1} Let $(e_1,e_2,e_3,e)$ be a basis of $G=C_2^3\oplus C_{2n}$ with $\ord(e_1)=\ord(e_2)=\ord(e_3)=2$ and $\ord(e)=2n$, where $n\geq 2$ is an even integer. Suppose that  $\theta:G\rightarrow G$ is the homomorphism defined by $\theta(e_1)=e_1$, $\theta(e_2)=e_2$, $\theta(e_3)=e_3$, $\theta(e)=ne$ and $\zeta:G\rightarrow G$ is the homomorphism defined by $\zeta(e_1)=\zeta(e_2)=\zeta(e_3)=0$, $\zeta(e)=e$.

 If  $S$ is a sequence of length $|S|=8$ over $G$ such that $\theta(S)$  is a squarefree sequence with $0\notin \mathsf{supp}(\theta(S))$,  then for any $k\in [1,n-1]$ and $\gcd(k,n)=1$, there exists a subsequence $T$ of $S$ with length $|T|\in [3,4]$ such that $\sigma(T)\in \ker(\theta)$ and $\sigma(T)\neq
  2ke$.
\end{lemma}

\begin{proof}Without loss of generality, we can assume that $k=1$. Otherwise choose $(e_1,e_2,e_3,ke)$ to be a basis of $G$.

   Assume to the contrary that for all subsequences $T$ of $S$ with $|T|\in [3,4]$ and  $\sigma(T)\in \ker(\theta)$, we have that  $\sigma(T)=
  2e$.

For any $v\in \theta(G)\setminus \{0\}$,
we define
$$\mathsf N_v(\theta(S))=|\{T\t \theta(S) \ :\  |T|=2 \text{ and } \sigma(T)=v  \}|+\delta_v,$$
where
\[
\delta_v =\left \{ \begin{array}{ll}  1 ,\quad \mbox{ if } v\in \mathsf{supp}(\theta(S)); \\  0,\quad \mbox{ if } v\notin \mathsf{supp}(\theta(S)).
\end{array} \right.
\]

Then $\sum_{v\in \theta(G)\setminus \{0\}}\mathsf N_v(\theta(S))=\frac{8\times 7}{2}+8=36$ and $|\theta(G)\setminus \{0\}|=15$ which implies that there exists an element $v\in \theta(G)\setminus \{0\}$ such that $\mathsf N_v(\theta(S))\ge3$. Therefore we can distinguish the following two cases.

\medskip\noindent{\bf Case 1.} There exists $v\in \theta(G)\setminus\{0\}$
 such that $\mathsf N_v(\theta(S))\ge 3$ and $\delta_v=1$.
 \smallskip

Without loss of generality, we can assume that   $\sigma(\theta(h_1))=\sigma(\theta(h_2h_3))=\sigma(\theta(h_4h_5))$. Then we have  $\sigma(\theta(h_1h_2h_3))=\sigma(\theta(h_2h_3h_4h_5))=\sigma(\theta(h_4h_5h_1))=0$ which implies that $\sigma(\zeta(h_1h_2h_3))=\sigma(\zeta(h_2h_3h_4h_5))=\sigma(\zeta(h_4h_5h_1))=2e$. Therefore $\zeta(h_1)=\zeta(h_2+h_3)=\zeta(h_4+h_5)=e$ or $(n+1)e$.

Let $\zeta(h_i)=k_ie$, where $k_i\in [0,2n-1]$ for each $i\in [1,8]$. Then $k_2+k_3,k_4+k_5$ are odd and there exist distinct $i,j\in [6,8]$ such that $k_i\equiv k_j\pmod 2$. Without loss of generality, we can assume that $k_2,k_4,k_6+k_7$ are even and hence $k_1,k_3,k_5$ are odd.
 Consider the sequence $W=h_1h_2h_4h_6h_7$ (see Figure \ref{ff3}).

  \begin{figure}[ht]
  \begin{center}
\setlength{\unitlength}{0.7 mm}%
\begin{picture}(102.66,34.98)(0,0)
\put(30.73,22.97){\circle*{1.80}}
\put(30.39,2.62){\circle*{1.80}}
\put(50.74,22.97){\circle*{1.80}}
\put(50.39,2.62){\circle*{1.80}}
\put(70.40,22.97){\circle*{1.80}}
\put(90.40,22.97){\circle*{1.80}}
\put(10.38,12.45){\circle*{1.80}}
\put(7.41,14.40){\fontsize{8.53}{10.24}\selectfont \makebox(7.5, 3.0)[l]{$h_1$\strut}}
\put(26.76,25.00){\fontsize{8.53}{10.24}\selectfont \makebox(7.5, 3.0)[l]{$h_2$\strut}}
\put(26.90,5.09){\fontsize{8.53}{10.24}\selectfont \makebox(7.5, 3.0)[l]{$h_3$\strut}}
\put(46.60,25.00){\fontsize{8.53}{10.24}\selectfont \makebox(7.5, 3.0)[l]{$h_4$\strut}}
\put(46.87,4.92){\fontsize{8.53}{10.24}\selectfont \makebox(7.5, 3.0)[l]{$h_5$\strut}}
\put(65.91,25.00){\fontsize{8.53}{10.24}\selectfont \makebox(7.5, 3.0)[l]{$h_6$\strut}}
\put(86.49,25.00){\fontsize{8.53}{10.24}\selectfont \makebox(7.5, 3.0)[l]{$h_7$\strut}}
\thinlines\lbezier(20.76,32.78)(100.47,32.78)\lbezier(100.47,32.78)(100.47,18.12)\lbezier(100.47,18.12)(22.50,18.12)\lbezier(22.50,18.12)(10.48,5.64)\lbezier(10.48,5.64)(2.00,13.93)\lbezier(2.00,13.93)(20.76,32.78)
\end{picture}%

  \end{center}
  \caption{}\label{ff3}
  \end{figure}

  Since $\theta(W)\in \mathscr{F}(\theta(G))$ and $\mathsf D(\theta(G))=\mathsf D(C_2^4)=5$, there exists a subsequence $V\mid W$ such that $\sigma(\theta(V))=0$ and $|V|\in \{3, 4, 5\}$. We distinguish three cases depending on $|V|$.

Suppose that $|V|=5$. Then $\sigma(\zeta(V))=(2k+1)e$ for some $k\in [0,n-1]$, a contradiction to $\sigma(\theta(V))=0$.

Suppose that $|V|=4$. If $h_2h_4\t V$, then $\sigma(\theta(V))=\sigma(\theta(V(h_2h_4)^{-1}h_3h_5))=\sigma(\theta(h_2h_3h_4h_5))=0$ and hence  $\sigma(\zeta(V))=\sigma(\zeta(V(h_2h_4)^{-1}h_3h_5))=\sigma(\zeta(h_2h_3h_4h_5))=2e$. Therefore $\zeta(h_2+h_4)=e$ or $(n+1)e$, a contradiction to $k_2,k_4$ are even. Thus, without loss of generality, we only need to consider $V=h_1h_2h_6h_7$. Then $\sigma(\zeta(V))=(2k+1)e$ for some $k\in [0,n-1]$, a contradiction to $\sigma(\theta(V))=0$.

Suppose that $|V|=3$. By symmetry, we only need to consider $V=h_1h_6h_7$, $V=h_2h_6h_7$, $V=h_2h_4h_6$ or $V=h_1h_2h_4$. If $V=h_1h_6h_7$, then $\sigma(\zeta(V))=(2k+1)e$ for some $k\in [0,n-1]$, a contradiction to $\sigma(\theta(V))=0$.
If $V=h_2h_6h_7$, then $\sigma(\theta(h_1h_3h_6h_7))=\sigma(\theta(V))=\sigma(\theta(h_1h_2h_3))=0$ and hence $\sigma(\zeta(h_1h_3h_6h_7))=\sigma(\zeta(V))=\sigma(\zeta(h_1h_2h_3))=2e$. It follows that $\zeta(h_2)=\zeta(h_1+h_3)=\zeta(h_6+h_7)=e$ or $(n+1)e$, a contradiction. If $V=h_2h_4h_6$, then $\sigma(\zeta(V))=\sigma(\zeta(h_3h_5h_6))=\sigma(\zeta(h_2h_3h_4h_5))=2e$ which implies that $\zeta(h_2+h_4)=e$ or $(n+1)e$, a contradiction to $k_2,k_4$ are even.  If $V=h_1h_2h_4$, then $\theta(h_3)=\theta(h_4)$, a contradiction.

\medskip\noindent{\bf Case 2.} There exists  $v\in \theta(G)\setminus\{0\}$ such that
 $\mathsf N_v(\theta(S))\ge 3$ and $\delta_v=0$.
 \smallskip

 Without loss of generality, we can assume that $\theta(h_1+h_2)=\theta(h_3+h_4)=\theta(h_5+h_6)$. Then  $\sigma(\theta(h_1h_2h_3h_4))=\sigma(\theta(h_3h_4h_5h_6))=\sigma(\theta(h_5h_6h_1h_2))=0$ and hence  $\sigma(\zeta(h_1h_2h_3h_4))=\sigma(\zeta(h_3h_4h_5h_6))=\sigma(\zeta(h_5h_6h_1h_2))=2e$. Therefore $\zeta(h_1+h_2)=\zeta(h_3+h_4)=\zeta(h_5+h_6)=e$ or $(n+1)e$.

 Let $\zeta(h_i)=k_ie$, where $k_i\in [0,2n-1]$ for each $i\in [1,8]$. Without loss of generality, we can assume that $k_2,k_4,k_6$ are even and $k_1,k_3,k_5$ are odd.

Therefore we can distinguish the following two cases.

\medskip\noindent{\bf Subcase 2.1.} $k_7,k_8$ are odd.
 \smallskip

Consider the sequence $W=h_1h_3h_5h_7h_8$ (see Figure \ref{f8}).
\begin{figure}[ht]
\begin{center}
\setlength{\unitlength}{0.7 mm}%
\begin{picture}(76.76,37.24)(0,0)
\put(7.00,27.00){\circle*{1.80}}
\put(7.00,7.00){\circle*{1.80}}
\put(27.00,27.00){\circle*{1.80}}
\put(27.00,7.00){\circle*{1.80}}
\put(47.00,27.00){\circle*{1.80}}
\put(47.00,7.00){\circle*{1.80}}
\put(67.00,27.00){\circle*{1.80}}
\put(87.00,27.00){\circle*{1.80}}

\put(3.65,28.85){\fontsize{8.53}{10.24}\selectfont \makebox(7.5, 3.0)[l]{$h_1$\strut}}
\put(3.71,9.67){\fontsize{8.53}{10.24}\selectfont \makebox(7.5, 3.0)[l]{$h_2$\strut}}
\put(23.90,28.85){\fontsize{8.53}{10.24}\selectfont \makebox(7.5, 3.0)[l]{$h_3$\strut}}
\put(23.55,9.67){\fontsize{8.53}{10.24}\selectfont \makebox(7.5, 3.0)[l]{$h_4$\strut}}
\put(43.44,28.85){\fontsize{8.53}{10.24}\selectfont \makebox(7.5, 3.0)[l]{$h_5$\strut}}
\put(44.01,9.67){\fontsize{8.53}{10.24}\selectfont \makebox(7.5, 3.0)[l]{$h_6$\strut}}
\put(63.52,28.85){\fontsize{8.53}{10.24}\selectfont \makebox(7.5, 3.0)[l]{$h_7$\strut}}
\put(84.09,28.85){\fontsize{8.53}{10.24}\selectfont \makebox(7.5, 3.0)[l]{$h_8$\strut}}
\thinlines
\lbezier(2.00,35.24)(94.76,35.24)
\lbezier(94.76,35.24)(94.76,22.48)
\lbezier(94.76,22.48)(2.00,22.48)
\lbezier(2.00,22.48)(2.00,35.24)
\end{picture}%

\end{center}
\caption{}\label{f8}
\end{figure}

Since $\theta(W)\in \mathscr{F}(\theta(G))$ and $\mathsf D(\theta(G))=\mathsf D(C_2^4)=5$, there exists a subsequence $V\mid W$ such that $\sigma(\theta(V))=0$ and $|V|\in \{3, 4, 5\}$. We distinguish two cases depending on $|V|$.

Suppose that $|V|=5$ or $3$. Then $\sigma(\zeta(V))=(2k+1)e$ for some $k\in[0,n-1]$, a contradiction to $\sigma(\theta(V))=0$.

Suppose that $|V|=4$. By symmetry, we only need to consider $V=h_1h_3h_5h_7$ or $V=h_1h_3h_7h_8$. For both cases, $h_1h_3\t V$. Since  $\sigma(\theta(V))=\sigma(\theta(V(h_1h_3)^{-1}h_2h_4))=\sigma(\theta(h_1h_2h_3h_4))=0$, we obtain that $\sigma(\zeta(V))=\sigma(\zeta(V(h_1h_3)^{-1}h_2h_4))=\sigma(\zeta(h_1h_2h_3h_4))=2e$ which implies that $\zeta(h_1+h_3)=e$ or $(n+1)e$, a contradiction to $k_1,k_3$ are odd.

\medskip\noindent{\bf Subcase 2.2.} $k_7$ or $k_8$ is even. Say, $k_7$ is even.
 \smallskip

Consider the sequence $W=h_1h_2h_4h_6h_7$ (see Figure \ref{f9}).
\begin{figure}[ht]
\begin{center}
\setlength{\unitlength}{0.7 mm}%
\begin{picture}(76.76,37.24)(0,0)
\put(7.00,27.04){\circle*{1.80}}
\put(6.96,6.69){\circle*{1.80}}
\put(27.31,27.04){\circle*{1.80}}
\put(26.96,6.69){\circle*{1.80}}
\put(46.97,27.04){\circle*{1.80}}
\put(67.52,6.69){\circle*{1.80}}
\put(3.65,28.85){\fontsize{8.53}{10.24}\selectfont \makebox(7.5, 3.0)[l]{$h_1$\strut}}
\put(3.71,9.67){\fontsize{8.53}{10.24}\selectfont \makebox(7.5, 3.0)[l]{$h_2$\strut}}
\put(23.90,28.85){\fontsize{8.53}{10.24}\selectfont \makebox(7.5, 3.0)[l]{$h_3$\strut}}
\put(23.55,9.67){\fontsize{8.53}{10.24}\selectfont \makebox(7.5, 3.0)[l]{$h_4$\strut}}
\put(43.44,28.85){\fontsize{8.53}{10.24}\selectfont \makebox(7.5, 3.0)[l]{$h_5$\strut}}
\put(44.01,9.67){\fontsize{8.53}{10.24}\selectfont \makebox(7.5, 3.0)[l]{$h_6$\strut}}
\put(64.09,9.67){\fontsize{8.53}{10.24}\selectfont \makebox(7.5, 3.0)[l]{$h_7$\strut}}
\put(47.05,6.69){\circle*{1.80}}
\thinlines
\lbezier(2.00,35.24)(13.00,35.24)
\lbezier(13.00,35.24)(13.00,15.00)
\lbezier(74.76,15.00)(13.00,15.00)
\lbezier(74.76,15.00)(74.76,2.00)
\lbezier(74.76,2.00)(2.00,2.00)
\lbezier(2.00,2.00)(2.00,35.24)
\end{picture}%

\end{center}
\caption{}\label{f9}
\end{figure}

Since $\theta(W)\in \mathscr{F}(\theta(G))$ and $\mathsf D(\theta(G))=\mathsf D(C_2^4)=5$, there exists a subsequence $V\mid W$ such that $\sigma(\theta(V))=0$ and $|V|\in \{3, 4, 5\}$. We distinguish three cases depending on $|V|$.

Suppose that $|V|=5$. Then $\sigma(\zeta(V))=(2k+1)e$ for some $k\in[0,n-1]$, a contradiction to $\sigma(\theta(V))=0$.

Suppose that $|V|=4$. Since $\sigma(\theta(V))=0$, we obtain that $V=h_2h_4h_6h_7$. Then $\sigma(\theta(V))=\sigma(\theta(h_1h_3h_6h_7))=\sigma(\theta(h_1h_2h_3h_4))=0$ and hence $\sigma(\zeta(V))=\sigma(\zeta(h_1h_3h_6h_7))=\sigma(\zeta(h_1h_2h_3h_4))=2e$. Therefore $\zeta(h_1+h_3)=e$ or $(n+1)e$, a contradiction to $k_1,k_3$ are odd.

Suppose that $|V|=3$. Since $\sigma(\theta(V))=0$, we obtain that $h_1\nmid V$. By symmetry, we only need to consider $V=h_2h_4h_6$ or $V=h_2h_4h_7$.  For both cases, $h_2h_4\t V$. Since  $\sigma(\theta(V))=\sigma(\theta(V(h_2h_4)^{-1}h_1h_3))=\sigma(\theta(h_1h_2h_3h_4))=0$, we obtain that $\sigma(\zeta(V))=\sigma(\zeta(V(h_2h_4)^{-1}h_1h_3))=\sigma(\zeta(h_1h_2h_3h_4))=2e$ which implies that $\zeta(h_1+h_3)=e$ or $(n+1)e$, a contradiction to $k_1,k_3$ are odd.
\end{proof}

\begin{prop}\label{RE2} $\eta(C_2^3\oplus C_{2n})=2n+6$, where $n\geq  2$ is an even integer.
\end{prop}

\begin{proof} Let $G=C_2^3\oplus C_{2n}$, where $n\ge 2$  is an even integer.  Suppose that  $\theta:G\rightarrow G$ is the homomorphism defined by $\theta(e_1)=e_1$, $\theta(e_2)=e_2$, $\theta(e_3)=e_3$, $\theta(e)=ne$ and $\zeta:G\rightarrow G$ is the homomorphism defined by $\zeta(e_1)=\zeta(e_2)=\zeta(e_3)=0$, $\zeta(e)=e$.

  In order to prove that $\eta(G)=2n+6$, by Lemma \ref{ETA} we only need to prove that $\eta(G)\le 2n+6$.
 Assume to the contrary that there exists a sequence  $S$ of length $2n+6$ over $G$  containing no short zero-sum subsequence.

Since $|\theta(S)|=|S|=2n+6=2(n-5)+16$, $\eta(C_2^4)=16$, and $\mathsf D(C_n)=n$,
we obtain that $S$ allows a product decomposition as
$$S=S_1\cdot \ldots \cdot S_{r}\cdot S_0,$$
where $S_1, \ldots, S_{r}, S_0$ are sequences over $G$ and, for every $i\in [1, r]$, $\theta(S_i)$ has sum zero and length $|S_i|\le 2$. What's more, $\theta(S_0)$ has no zero-sum subsequence of length $\le 2$  and $n-4\leq r\leq n-1$.
We distinguish the following four cases depending on $r$ to get contradictions.

\medskip

\noindent\textbf{Case 1.} $r=n-1$. Then $|S_0|\ge 8$.
\smallskip

Since $S$ has no short zero-sum subsequence, we obtain that
 $\sigma(S_1)\cdot \ldots\cdot \sigma(S_{n-1})$ is zero-sum free over $\ker(\theta)\cong C_n$. By Lemma \ref{cyclic}.1,  there exists an element $k\in [1,n-1] $ such that  $\sigma(S_1)=\cdots =\sigma(S_{n-1})=2ke$ and $\gcd(k,n)=1$. Without loss of generality, we can assume that $k=1$.

Since $\theta(S_0)$ has no zero-sum subsequence of length $\le 2$, by  Lemma \ref{EIMP1} we obtain that there exists a subsequence $V$ of $S_0$ with length $|V|\in [3,4]$ such that $\sigma(V)=2te$ where $t\in [2,n]$.
 Then by calculation we get
$$\sigma(S_1\cdot\ldots\cdot S_{n-t}\cdot V)=0, \quad \text{ and }$$
$$|S_1\cdot\ldots\cdot S_{n-t}\cdot V|
\le 2(n-t)+4
\le 2n,$$ which implies that $S_1\cdot\ldots\cdot S_{n-t}\cdot V$ is a short zero-sum subsequence of $S$,
a contradiction.

\medskip

\noindent\textbf{Case 2.} $r=n-2$. Then $|S_0|\ge 10$.
\smallskip

Since $\theta(S_0)$ is squarefree, by Lemma \ref{SUM}, $S_0$ has  a subsequence  $T$ of length $3$  such that $\sigma(\theta(T))=0$. By our assumption,  the sequence $\sigma(S_1)\cdot \ldots\cdot \sigma(S_{n-2})\sigma(T)$ is zero-sum free over $\ker(\theta)\cong C_n$.  By Lemma \ref{cyclic}.1, we obtain that
$$\sigma(S_1)=\cdots=\sigma(S_{n-2})=\sigma(T)=2ke, \quad \text{ for some $k\in [1,n-1]$ and $\gcd(k,n)=1$.}$$

Without loss of generality, we can assume that $k=1$.
By Lemma \ref{EIMP1}, $S_0$ has  a subsequence  $T'$ of length $|T'|\in \{3,4\}$  such that $\sigma(\theta(T'))=0$ and $\sigma(T')=2te$ with $t\in [2,n]$. Therefore the sequence $S_1\cdot\ldots\cdot S_{n-t}\cdot T'$ is a short zero-sum subsequence of $S$, a contradiction.

\medskip

\noindent\textbf{Case 3.} $r=n-3$. Then $|S_0|\ge 12$ and $n\ge 4$.
\smallskip

\smallskip

By Lemma \ref{SUM}, there exist two subsequences $T_1,T_2$ of  $ S_0$ such that $\sigma(\theta(T_1))=\sigma(\theta(T_2))=0$ and $|T_1|=|T_2|=3$.

 By our assumption, the sequence $\sigma(S_1)\cdot \ldots\cdot \sigma(S_{n-3})\cdot\sigma(T_1)\cdot \sigma(T_2)$ contains no zero-sum subsequence. Therefore by Lemma \ref{cyclic}.1, there exists an  element $k\in [1,n-1]$ and $\gcd(k,n)=1$ such  that
$$\sigma(S_1)\cdot\ldots\cdot\sigma(S_{n-3})\cdot\sigma(T_1)\cdot\sigma(T_2)=(2ke)^{n-1}.$$
Without loss of generality, we can assume  that  $k=1$.
By Lemma \ref{EIMP1}, there exists  $T_1'\mid S_0$ such that $\sigma(\theta(T_1'))=0$, $|T_1'|\in [3,4]$, and $\sigma(T_1')=2t_1e$ with $t_1\in[2,n]$.
By Lemma \ref{EIMP1} again, there exists  $T_2'\mid S_0(T_1')^{-1}$ such that $\sigma(\theta(T_2'))=0$, $|T_2'|\in [3,4]$, and $\sigma(T_2')= 2t_2e$ with $t_2\in[2,n]$.

If $t_1\ge 3$, then $S_1\cdot\ldots\cdot S_{n-t_1}\cdot T_1'$ is a short zero-sum subsequence of $S$, a contradiction.
If $t_2\ge 3$, then $S_1\cdot\ldots\cdot S_{n-t_2}\cdot T_2'$ is a short zero-sum subsequence of $S$, a contradiction.
Otherwise $t_1+t_2=4\le n$. Then $S_1\cdot\ldots\cdot S_{n-4}\cdot T_1'\cdot T_2'$ is a short zero-sum subsequence of $S$, a contradiction.

\medskip
\noindent\textbf{Case 4.} $r=n-4$. Then $|S_0|\ge 14$ and $n\ge 4$.
\smallskip

We distinguish two cases depending on $n$.

\noindent\textbf{Subcase 4.1.} $n\ge 6$.
\smallskip

By Lemma \ref{SUM}, there exist two disjoint  subsequences $T_1,T_2$ of  $ S_0$ such that $\sigma(\theta(T_1))=\sigma(\theta(T_2))=0$ and $|T_1|=|T_2|=3$.

 By our assumption, the sequence $\sigma(S_1)\cdot \ldots\cdot \sigma(S_{n-4})\cdot\sigma(T_1)\cdot \sigma(T_2)$ contains no zero-sum subsequence. Therefore by Lemma \ref{cyclic}.2, there exists an  element $k\in [1,n-1]$ and $\gcd(k,n)=1$ such  that
$$\sigma(S_1)\cdot\ldots\cdot\sigma(S_{n-4})\cdot\sigma(T_1)\cdot\sigma(T_2)=(2ke)^{n-3}4ke\text{ or }(2ke)^{n-2}.$$
Without loss of generality, we can assume  that  $k=1$ and $\sigma(T_1)=2e$. By Lemma \ref{EIMP1}, there exists  $T_3\mid S_0(T_1)^{-1}$ such that $\sigma(\theta(T_3))=0$, $|T_3|\in [3,4]$, and $\sigma(T_3)\neq 2e$. Then the sequence $\sigma(S_1)\cdot \ldots\cdot \sigma(S_{n-4})\cdot\sigma(T_1)\cdot \sigma(T_3)$ contains no zero-sum subsequence and hence $\sigma(S_1)\cdot \ldots\cdot \sigma(S_{n-4})=(2e)^{n-4}$.

By Lemma \ref{EIMP1}, there exists  $T_1'\mid S_0$ such that $\sigma(\theta(T_1'))=0$, $|T_1'|\in [3,4]$, and $\sigma(T_1')=2t_1e$ with $t_1\in[2,n]$.
By Lemma \ref{EIMP1} again, there exists  $T_2'\mid S_0(T_1')^{-1}$ such that $\sigma(\theta(T_2'))=0$, $|T_2'|\in [3,4]$, and $\sigma(T_2')= 2t_2e$ with $t_2\in[2,n]$.

If $t_1\ge 4$, then $S_1\cdot\ldots\cdot S_{n-t_1}\cdot T_1'$ is a short zero-sum subsequence of $S$, a contradiction.
If $t_2\ge 4$, then $S_1\cdot\ldots\cdot S_{n-t_2}\cdot T_2'$ is a short zero-sum subsequence of $S$, a contradiction.
Otherwise $t_1+t_2\le 6\le n$. Then $S_1\cdot\ldots\cdot S_{n-t_1-t_2}\cdot T_1'\cdot T_2'$ is a short zero-sum subsequence of $S$, a contradiction.

\smallskip\noindent\textbf{Subcase 4.2.} $n=4$. Then $S=S_0$.
\smallskip

  By Lemma \ref{SUM}, there exist two disjoint  subsequences $T_1,T_2$ of  $ S_0$ such that $\sigma(\theta(T_1))=\sigma(\theta(T_2))=0$ and $|T_1|=|T_2|=3$.

 If $\sigma(T_1)\neq \sigma(T_2)$, since $T_1T_2$ can not be zero-sum, without loss of generality, we can assume that $\sigma(T_1)=2e$ and $\sigma(T_2)=4e$.  By Lemma \ref{EIMP1}, there exists  $T_3\mid S_0(T_1T_2)^{-1}$ such that $\sigma(\theta(T_3))=0$, $|T_3|\in [3,4]$, and $\sigma(T_3)= 2te$ with $t\in[2,4]$. Thus, one of the sequences $T_3, T_1T_3, T_2T_3$ must be a short zero-sum subsequence of $S$, a contradiction.

Then $\sigma(T_1)=\sigma(T_2)$, since $T_1T_2$ can not be zero-sum, without loss of generality, we can assume that $\sigma(T_1)=\sigma(T_2)=2e$.

 We claim that for any subsequence $T$ of $S$ satisfying that $|T|=3$ and $\sigma(\theta(T))=0$, we have $\sigma(T)=2e$.

 In fact, $T_1$ or $T_2$ must be disjoint with $T$. We can assume that $T_1$ and $T$ are disjoint. If $\sigma(T)=6e$, then $T_1T$ is a short zero-sum subsequence, a contradiction. If $\sigma(T)=4e$, we can do it as before to obtain a contradiction. Then $\sigma(T)=2e$.

Since $\sigma(T_1)=2e$, we can
 choose $g\t T_1$ such that $\zeta(g)\not\in \{e,5e\}$.  By Lemma \ref{SUM}, there exist subsequences $R_1,\ldots, R_4$ of $ST_1^{-1}$ such that $\theta(g)=\sigma(\theta(R_1))=\ldots=\sigma(\theta(R_4))$ and $|R_1|=\ldots=|R_4|=2$. Since for each $i\in [1,6]$, $\sigma(\theta(gR_i))=0$, we obtain that $\sigma(gR_i)=2e$. Thus $\sigma(\zeta(R_i))=2e-\zeta(g)$ for each $i\in [1,4]$. By $\sigma(\theta(R_1R_2))=\sigma(\theta(R_1R_2))=0$, we have $\sigma(R_1R_2)=\sigma(R_1R_2)=4e-2\zeta(g)$. If $\sigma(R_1R_2)=2e$, then $\zeta(g)\in \{e,5e\}$, a contradiction. If $\sigma(R_1R_2)=4e$, then $R_1R_2R_3R_4$ is a short zero-sum subsequence of $S$, a contradiction. Otherwise $\sigma(R_1R_2)=6e$. Then $T_1R_1R_2$ is a short zero-sum subsequence of $S$, a contradiction.
\end{proof}

\begin{proof}[\bf Proof of Theorem \ref{Th1}.2]

By Proposition \ref{RE1} and \ref{RE2}, it follows  that
 $\eta(G)=2n+6$.
If $n\geq 36=\max\{2|C_2^3|+1, 4|C_2^3|+4\}$,  by Lemma \ref{ETAF},  we have that $\mathsf s(C_2^3\oplus C_{2n})=\eta(C_2^3\oplus C_{2n})+{\exp}(C_2^3\oplus C_{2n})-1=2n+6+2n-1=4n+5$.
\end{proof}

\subsection*{Acknowledgements}
The authors would like to thank Professor Alfred Geroldinger of University of Graz for his many helpful suggestions.  
 This research was supported by NSFC (grant no. 11401542), the Fundamental Research Funds for the Central Universities (grant no. 2652014033), and
the Austrian Science Fund FWF (project no.  M1641-N26).

\bigskip


\end{document}